\documentclass[12pt]{amsart}
\usepackage{graphicx}
\usepackage[dvipsnames]{xcolor}

\usepackage[abbrev]{amsrefs}

\newtheorem{thm}{Theorem}[section]
\newtheorem{cor}[thm]{Corollary}
\newtheorem{lem}[thm]{Lemma}
\newtheorem{prop}[thm]{Proposition}
\theoremstyle{definition}

\theoremstyle{remark}

\numberwithin{equation}{section}
\newcommand{\norm}[1]{\left\Vert#1\right\Vert}
\newcommand{\inner}[1]{\langle #1 \rangle}

\newcommand{\set}[1]{\left\{#1\right\}}

\newcommand{\ddbar}{\sqrt{-1}\,\bar{\partial}\partial}

\newcommand{\Lie}[1]{\mathfrak{#1}}

\newcommand{\Hilb}{\mathrm{Hilb}_{k}}
\newcommand{\FS}{\mathrm{FS}_{k}}
\newcommand{\tr}{\mathrm{tr}}
\usepackage{marginnote}  %julien added
\usepackage{enumitem} %julien added
\usepackage[colorlinks=true]{hyperref}  %julien added

\newtheorem{Thm}{Theorem}[section]

\newtheorem{Def}[Thm]{Definition}

\newtheorem{Rmk}{Remark}

\begin{document}
\title[]{Quantization of Donaldson's heat flow over projective manifolds}%

 \author{Julien Keller} \address{Institut de Math\'ematiques de Marseille I2M, UMR 7373, Aix - Marseille Universit\'e,
39 rue F. Joliot-Curie, 13453 Marseille Cedex 13,  France}
 \email[Julien Keller]{julien.keller@univ-amu.fr}

 \author{Reza Seyyedali} \address{Department of Pure
Mathematics, University of Waterloo,
Waterloo, Ontario, N2L 3G1, Canada}
\email[Reza Seyyedali]{rseyyeda@uwaterloo.ca}



\date{October 17, 2014}
\begin{abstract}
Consider $E$ a holomorphic vector bundle over a projective manifold $X$ polarized by an ample line bundle $L$. Fix $k$ large enough, the holomorphic sections $H^0(E\otimes L^k)$ provide embeddings of $X$ in a Grassmanian space. We define the \textit{balancing flow for bundles} as a flow on the space of projectively equivalent embeddings of $X$. This flow can be seen as a flow of algebraic type hermitian metrics on $E$. At the quantum limit $k\to \infty$, we prove the convergence of the balancing flow towards the Donaldson heat flow, up to a conformal change. As a by-product, we obtain a numerical scheme to approximate the Yang-Mills flow in that context.
\end{abstract}
\maketitle
\tableofcontents
% ----------------------------------------------------------------
%\pagestyle{plain}

%%%%%%%%%%%%%%%%
%\makeatletter  %julien added
%   \providecommand\@dotsep{5}
% \makeatother
% \listoftodos\relax   % list of TODO things
%%%%%%%%%%%% 
% in order to use TODO in align environment
%\makeatletter
%\renewcommand{\@todonotes@drawMarginNoteWithLine}{%
%\begin{tikzpicture}[remember picture, overlay, baseline=-0.75ex]%
%    \node [coordinate] (inText) {};%
%\end{tikzpicture}%
%\marginnote[{% Draw note in left margin
%    \@todonotes@drawMarginNote%
%    \@todonotes@drawLineToLeftMargin%
%}]{% Draw note in right margin
%    \@todonotes@drawMarginNote%
%    \@todonotes@drawLineToRightMargin%
%}%
%}
%\makeatother
%%%%%%%%%%%%%%%%

\section{Introduction}

Let $E$ be a holomorphic vector bundle over
a compact K\"ahler manifold $X$. The Yang-Mills flow provides a heat flow approach to the
Hermitian-Einstein problem over $E$, which was first solved in the case of curves by M.S Narasimhan
and C.S. Seshadri, by S.K. Donaldson for projective manifolds \cites{Do85,Do87}, and then by K. Uhlenbeck and S.T. Yau \cite{UY}
for K\"ahler manifolds. A crucial ingredient of Donaldson's proof is the precise study of a heat flow related to the famous Yang-Mills flow. S.K. Donaldson showed that the convergence of the flow is equivalent to an algebraic stability condition. This
heat flow approach has been extended to various other settings, including reflexive sheaves or Higgs bundles. When the bundle $E$ is polystable, Donaldson heat flow converges towards a Hermitian-Einstein metric.  When $E$ is not stable, the flow develops at its limit some interesting bubbling phenomena involving the graded Harder-Narasimhan filtration of $E$. Therefore, this flow reveals a lot of information on the geometric structure of the underlying bundle. \\
In this paper, we provide a quantization of Donaldson heat flow of endomorphisms when $X$ is assumed to be a complex projective manifold. Building on the formalism of moment maps, we construct a flow of hermitian metrics on the Bergman space $\mathcal{B}_k$ naturally associated to our data and the quantization parameter $k>>0$. Since we are dealing with a compact manifold, the Bergman space is a finite dimensional space and the symplectic formalism has an interpretation in terms of Geometric Invariant Theory. We are calling \textit{balancing flow} this flow of algebraic type hermitian metrics over the bundle $E$. \\
At the quantum limit when $k\rightarrow +\infty$, we prove that the balancing flow converges in $C^1$ topology towards the Donaldson heat flow (Theorem \ref{thm1} and its corollaries). In particular, this provides a numerical scheme to approximate  Donaldson heat flow (Theorem \ref{thm2}) and the Yang-Mills flow (see Remark \ref{rmk1}). \\
The idea of approximating metrics by using projective embeddings goes back several years. For instance, G. Tian proved that the
Bergman metrics are dense in the space of all K\"ahler metrics. In the same vein, it has been studied which of these algebraic metrics are canonical, starting with \cite{BLY}.
These crucial ideas were used later by S.K. Donaldson in the study of the constant scalar curvature problem for K\"ahler metrics, in relation with the so-called Yau-Tian-Donaldson conjecture, see \cite{Do2}. In \cite{Fi1}, J. Fine introduced another balancing flow that approximates the Calabi flow, a 4th order parabolic PDE that is expected to deforms a given K\"ahler metric towards a constant scalar curvature one. In \cite{C-K} it was introduced the $\Omega$-balancing flow and shown that it converges towards a flow of K\"ahler metrics that enjoys similar properties to the K\"ahler-Ricci flow, providing a new approach to the classical Calabi conjecture. \\
Switching from the line bundle case to the case of bundles of any rank, X. Wang has been studying the Hermitian-Einstein equation in \cite{W1}, \cite{W2} using again projective embeddings. He introduced the notion of balanced metrics for a bundle $E$ and proved that there existence is equivalent to the Gieseker stability of $E$. Furthermore, he obtained their convergence at the quantum limit towards a (weakly) Hermitian-Einstein metric when this latter does exist a priori. Later, the second author proved that Wang balanced metrics are fixed attractive points of a certain natural operator on the Bergman space $\mathcal{B}_k$, see \cite{Se}. Our work can be seen as complementary of the results of X. Wang and  R. Seyyedali. \\
Let us explain briefly how is organized this paper. In Section \ref{sect2}, we provide a brief introduction to the work of X. Wang and define balanced metrics for bundles using embeddings in Grassmanians. In section \ref{sect3}, we define precisely the balancing flow and study its behavior at the quantum limit when one assumes it enjoys a limit.
Section  \ref{sect4} contains the proofs of the main results. A technical part of these proofs relies on the asymptotics for Bergman kernels and Bergman functions, the so-called Catlin-Tian-Yau-Zelditch expansion (in short CTYZ) that has been extended for bundles (see \cite[Chapter IV, Theorem 4.1.2]{MM} for an extensive reference on the topic, and also \cites{Ca,Z,Lu}). Another technical ingredient is an asymptotics of the $Q_k$ operator introduced in \cite{Do1} adapted to the bundle case (Theorem \ref{Qkthm}) from the earlier work of K. Liu and X. Ma in \cite{Fi1}. To obtain the main result (Theorem \ref{thm1}), we use a deformation argument (Sections \ref{deform1} and \ref{deform2}), inspired from \cite{Do2}. We also use the fact that the balancing flow is  
the downward gradient of the norm square of a certain moment map. In Section \ref{estimates}, we derive some projective estimates to control uniformly the curvature tensors of the involved metrics using the Riemannian distance on the Bergman space. The last section, Section \ref{mainsect} contains all the main results of the paper.\\
 We believe that there is some flexibility with the choice of the heat equation we studied, i.e that the balancing flow could be adapted to a more general context of decorated vector bundle, following the lines of \cite{Ke}. Eventually, we expect that the main result of this paper can be developed in a more general setting  (for instance on complete hermitian bundle, compact hermitian manifold equipped with big line bundles) since the  required asymptotics results also hold in a more general setting. This is a direction that could be investigated in the future using the recent developments on the subject.\\

{\small
{\bf Acknowledgements.} The first author is very grateful to Simon Donaldson, Joel Fine and Julius Ross for illuminating conversations on the subject of balanced embeddings throughout the years. He would also like to thank Xiaonan Ma and Xiaowei Wang.}  

\section{A quick review of the Fubini-Study geometry on Grassmanians}

Denote the space of all matrices $z \in \mathcal{M}_{r \times N}(\mathbb{C})$ with rank $r$ by  $\mathcal{M}^0_{r \times N}.$ By definition,
$$G(r,N)=\frac{\mathcal{M}^0_{r \times N}}{\sim},$$ where $z \sim w$ if and only if there exists $P \in GL(r, \mathbb{C})$  such that $z=Pw.$ The tangent bundle of $G(r,N) $ is given by $$\frac{\set{(z,X)| z \in \mathcal{M}^0_{r \times N} , X\in \mathcal{M}_{r \times N}}}{\sim^{\prime}},$$ where $(z,X)\sim^{\prime} (w,Y)$ if and only if there exists $P \in GL(r, \mathbb{C})$ and $Q \in \mathcal{M}_{r \times r}(\mathbb{C})$ such that $$z=Pw, \,\,\, X=PY+Qw.$$
The Fubini-Study metric on $T G(r,N)$ is given by $$\inner{[(z,X)], [(z,Y)]}_{FS}=\tr(Y^*(zz^*)^{-1}X)-\tr(zY^*(zz^*)^{-2}Xz^*).$$
One can easily check that this is the Fubini-Study metric induced on $G(r,N)$ using the Plucker embedding.

\begin{Def}

Let $A \in \sqrt{-1}\Lie{u}(N)$. It induces a holomorphic vector field $\zeta_{A}$ on $G(r,N)$ as follows:

\begin{equation}\label{zeta} \zeta_{A}(z):=[z,zA].\end{equation}

%\color{blue}We also define $H_{A} \in C^{\infty}(G(r,N),End(U_{r}))$ by $$H_{A}=A\mu.$$
%Note that $tr(H_{A})$ is the hamiltonian associated to the vector field $\zeta_{A}$.\color{black}
%\todo{\color{red} I will have to move this after we define $\mu$ and $U_{r}$.\color{black}}

\end{Def}

We recall briefly the symplectic framework described in \cites{C-K,W1}.
Let us consider first $$\mu : G(r,N) \rightarrow \sqrt{-1} \mathfrak{u}(N)$$ the moment map associated to the $U(N)$ action and the Fubini-Study metric $\omega_{FS}$ on $G(r,N)$. Note that here we identify implicitly the Lie algebra $\sqrt{-1} \Lie{u}(N)$ with its dual using the Killing form $\inner{A,B}=\tr(AB)$. Given homogeneous unitary coordinates, then one has explicitly the following formula for $\mu$,
\begin{equation}
\mu([z])=z^*(zz^*)^{-1}z. \label{mu}
\end{equation}

\section{Context and Preliminaries\label{sect2}}

\subsection{Balanced metrics for bundles }
Let $(X,\omega)$ be a compact K\"ahler manifold of dimension
$n$ and $(L,\sigma)$ be an ample holomorphic hermitian line bundle over
$X$ such that its curvature satisfies $\ddbar{\sigma}=\omega.$ Let $E$ be a holomorphic vector
bundle of rank $r $ and degree $d$ over $X$. Since $L$ is an ample line bundle, using holomorphic sections of $H^{0}(X,E(k) )$, we can embed $X$ into $G(r, H^{0}(X,E(k))^*)$ for $k \gg 0$. Here $E(k)=E \otimes L^{\otimes k}$.
Indeed, for any $x \in X$, we have the evaluation map
$H^{0}(X,E(k))\rightarrow E(k)_{x}$, which sends $s$ to $s(x)$. Since
$E(k)$ is globally generated, this map is a surjection. So its dual
is an inclusion of $E(k)_{x}^* \hookrightarrow H^{0}(X,E(k))^*$, which
determines an $r$-dimensional subspace of $H^{0}(X,E(k))^*$.
Therefore we get a map $\iota: X \rightarrow G(r, H^{0}(X,E(k))^*)$.
Since $E(k)$ is very ample for $k \gg 0$, $\iota$ is an embedding. Clearly we have
$\iota^*U_{r}=E(k)^*$, where $U_{r}$ is the tautological vector
bundle on $G(r, H^{0}(X,E(k))^*)$, i.e. at any $r$-plane in $G(r,
H^{0}(X,E(k))^*)$, the fibre of $U_{r}$ is exactly that $r$-plane. Any
choice of basis for $H^{0}(X,E(k))$ gives an isomorphism between
$G(r, H^{0}(X,E(k))^*)$ and the standard $G(r,N_{k})$, where $$N_{k}= \dim
H^{0}(X,E(k)).$$ We have the standard Fubini-Study hermitian metric on
$U_{r}$, so we can pull it back to $E(k)$ and get a hermitian metric
on $E(k)$.

Let fix the embedding $\iota: X \hookrightarrow   G(r,N)$ and identify $X$ with its image $\iota(X)$. The space of Fubini-Study type metrics on the space $H^0(X,U_{r})=H^0(G(r,N),U_{r})$ is identified with the Bergman space $$\mathcal{B}=\mathcal{B}_k={GL(N)}/{U(N)}\simeq \sqrt{-1} \Lie{u}(N).$$ 
Then, we can consider the integral of $\mu$ over  $X$ with respect to the volume form $\displaystyle \frac{\omega^n}{n!}$: 
$$\bar{\mu}(\iota)=\int_X \mu(\iota(p)) \frac{\omega^n(p)}{n!} $$which induces a moment map for the $U(N)$ action over the space of all bases of $H^0(M,U_{r})=H^0(G(r,N),U_{r})$. More precisely on the space $\mathfrak{M}$ of smooth maps from $X$ to $ Gr(r,H^0(\iota^* U(r))^*)$, we have a natural symplectic structure $\varpi$ defined by
$$\varpi(a,b)=\int_X \inner{a,b}\frac{\omega^n}{n!},$$ for $a,b\in T_{\iota}\mathfrak{M}$ and $\inner{.,.}$ the Fubini-Study inner product induced on the tangent vectors.
Then, $U(N)$ acts isometrically on $\mathfrak{M}$ with the moment map given by $$\iota\mapsto-\sqrt{-1}\left(\bar{\mu}(\iota) - \frac{\tr(\bar{\mu}(\iota))}{N}Id_{N}\right)\in \sqrt{-1}\mathfrak{su}(N).$$ Note that using a hermitian metric $H$ on $H^0(X,\iota^*U_{r})$, one can consider an orthonormal basis with respect to $H$ and the associated embedding,  and thus it also makes sense to speak of $\bar{\mu}(H)$. 

\medskip

In the Bergman space $\mathcal{B}=GL(N)/U(N)$, we have a preferred metric associated to the volume form $\displaystyle \frac{\omega^n}{n!}$ and the moment map we have just defined, and this is precisely an balanced metric. By definition, a metric $H\in Met(H^0(X,\iota^*U_{r}))$ 
is said to be balanced if the trace free part of $\bar{\mu}(H)$ vanishes.

Let $\mathcal{K}$ be the space of hermitian metrics on $E$ and $\mathcal{B}_{k}$ be the space of hermitian inner products on $H^{0}(X,E(k))$. Following Donaldson \cite{Do1}, we can define the following maps:

\begin{itemize}

\item Define $$\Hilb:  \mathcal{K}\rightarrow \mathcal{B}_{k}$$ by

 $$\langle s,t \rangle _{\Hilb(h)}= \frac{N_{k}}{rV} \int_{X}
\langle s(x),t(x) \rangle_{h \otimes \sigma^k }\frac{\omega^n}{n!},$$
for any $s,t \in H^{0}(X,E(k))$ where $N_{k}=\dim(H^{0}(X,E(k))) $ and $V=\textrm{  Vol}(X,\omega)$. Note that
$\Hilb$ only depends on the volume form $\displaystyle \frac{\omega^{n}}{n!}$.

\item For the metric $H$ in $\mathcal{B}_{k}$, $\FS(H)$ is the unique metric
on $E$ such that  $$\sum s_{i}\otimes s_{i}^{*_{\FS(H) \otimes \sigma^k}}=Id_E,$$ where
$s_{1},...,s_{N_{k}}$ is an orthonormal basis for $H^{0}(X,E(k))$ with
respect to $H$. This gives the map $\FS:\mathcal{B}_{k} \rightarrow \mathcal{K}$.

\item Define a map $$\Phi_{k}: \mathcal{K} \rightarrow  \mathcal{K}$$ by $\Phi_{k}(h)=\FS \circ \Hilb(h).$
\end{itemize}

 It is not difficult to see that a balanced metric $H \in \mathcal{B}_k$ is a fixed point of the map $\Hilb \circ \FS$ map and in that case $\FS(H)$ is a fixed point of the map $\Phi_k$.

\subsection{The \texorpdfstring{$\Omega$}{Qk} operator}

Let $h$ be a hermitian metric on $E$ and consider $\{s_{1},\dots,s_{N_{k}}\}$ an orthonormal basis for
$H^{0}(X,E(k))$ with respect to $\Hilb(h),$ i.e. $$\frac{N_{k}}{rV} \int_{X}
\langle s_{i},s_{j} \rangle_{h}\frac{\omega^n}{n!}=\delta_{ij}.$$ Following Donaldson, we define the kernel $K_{k}(x,y): End(E_{x}) \rightarrow End(E_{y})$ and the operator $Q_{k}:\Gamma(End(E)) \rightarrow \Gamma(End(E))$ as follows. For any $\phi_{x} \in End(E_{x})$,
\begin{align*}K_{k}(x,y)\phi_{x}=&k^{n}\sum_{i,j=1}^{N_{k}} \langle s_{i}(x),\phi_{x}(s_{j}(x)) \rangle_{h\otimes \sigma^k}\langle .,s_{i}(y) \rangle_{h\otimes \sigma^k}s_{j}(y)\\
Q_{k}(\phi)(y)=&\int_{X}K_{k}(x,y)(\phi(x))\frac{\omega^n_x}{n!}.\end{align*}
Note that the integral is well defined, since we only integrate with respect to $x$. So we only integrate functions!

\begin{thm}\label{Qkthm}
 For any $m\geq 0$, there exists a constant $c>0$ such that for any $f\in \Gamma(End(E))$, 
\begin{align}\label{asymptQ2}\Big\Vert Q_{k}(f) - f\Big\Vert_{C^m} \leq \frac{c}{k}\Vert f \Vert_{C^m}\end{align}
Moreover, \eqref{asymptQ2} is uniform in the sense that there is an integer $s_0$ such that if all the data $\sigma,h$ run over a bounded set in $C^{s_0}$ topology then the constant $c$ is independent of $\sigma,h$.
\end{thm}

\begin{proof}
The proof is already contained in \cite[Appendix]{Fi1} and \cite{LM} where the case of $E$ trivial is treated in details.  The important point is that the full Bergman kernel expansion holds in our setting, see \cite[Theorem 4.18']{DLM}. The involved Dirac operator is 
$$\sqrt{2}\left( \bar\partial^{E\otimes L^k} + \bar\partial^{E\otimes L^k, *_\omega}\right)$$
whose kernel is identified with $H^0(E\otimes L^k)$. 
Let $P_k(x,x')$ where $(x,x')\in X\times X$, be the smooth kernel of the orthogonal projection from
$\Gamma(E\otimes L^k)$ onto $H^0(X,E\otimes L^k)$. Then the following results hold (the key point being the spectral gap property of the Dirac operator that allows to localize the problem). \\
\begin{itemize}[leftmargin=*]
\item There exist $J_r(Z,Z')$ polynomials such that for any integers $p,m,m'$, there exists $N\in \mathbb{N}, C>0, C_0>0$ such that for $Z,Z'\in T_{x_0}X$, $\vert Z\vert \leq \epsilon, \vert Z'\vert \leq \epsilon$, $x_0\in X$, one knows the asymptotic behavior of $P_k$ locally around $x_0\in X$, namely in normal coordinates as in \cite{LM},
\begin{align}
\Big\Vert \frac{1}{k^{2n}}P_{k,x_0}(Z,Z')- \sum_{r=0}^p (J_r\cdot P) (\sqrt{k}Z,\sqrt{k}Z')k^{-r/2}\Big\Vert_{C^{m'}} \nonumber\\
\leq \frac{C}{k^{(p+1)/2}}(1+\vert \sqrt{k}Z\vert + \vert \sqrt{k}Z'\vert)^Ne^{-C_0 \sqrt{k}\vert Z-Z'\vert} + O(k^{-\infty}) \label{local2}
\end{align}
 where the term $ O(k^{-\infty})$ means that the difference is dominated by $C_k k^{-l}$ for any $l>0$. \\
\item It is possible to identify the first terms of the expansion, 
\begin{equation}\label{J2} J_0= Id_E, \,\,\, J_1(Z,Z')=0,\end{equation}
and by $P(Z,Z')$ we mean the local Bergman kernel that is computed at \cite[(4.114)]{DLM}, \cite[Section 1]{MM2}. It is actually the classical heat kernel on $\mathbb{C}^n$, i.e \begin{equation}\label{heat2}\vert P(z,z')\vert^2=e^{-\pi\vert z-z'\vert^2}.\end{equation}
\item There is an exponential decrease of $P_k$ outside of the diagonal, i.e for any integers $l,m$, $\epsilon>0$, there exists $C>0$ such that
for $k>0$, $(x,x')\in X\times X$, 
\begin{equation}\label{decrease2}\Vert P_k(x,x')\Vert_{C^m(M\times M)} \leq Ck^{-l}
\end{equation}
if $d(x,x')>\epsilon$.
\item Given $f\in \Gamma(E)$, one can denote it as $f_{x_0}(Z)$ as endomorphism of $Z$, in normal coordinates around $x_0$.
\begin{equation}\Big\Vert   k^n \int_{\vert Z \vert \leq \epsilon} e^{-\pi k \vert Z\vert^2} f_{x_0}(Z) dV(Z) - f(x_0)  \Big\Vert_{C^m} \leq \frac{C}{k} \Vert f \Vert_{C^m}\label{op2}.\end{equation}
\end{itemize}

 From \eqref{local2}, \eqref{J2}, \eqref{decrease2}, \eqref{op2}, one gets by triangle inequality,
\begin{equation*}
 \Big\Vert\frac{1}{k^n}K_k(f)- f\Big\Vert_{C^m}
 \leq \frac{C}{k}\Big\Vert f \Big\Vert_{C^m}.
\end{equation*}
For the last part of the theorem, one notices that with our assumptions all the constants are uniformly bounded, see \cite[Section 4.5]{DLM} or \cite{LM}.
\end{proof}

We draw a direct consequence of the previous result.
 
\begin{prop}\label{prop1a}
For any $h \in  \mathcal{K}$ and any $\phi \in \Gamma(X,E) $ hermitian with respect to $h$, we have

\begin{enumerate}
\item $\Phi_{k}(h)= \frac{N_{k}}{rV} h B_{k}(h)^{-1}.$

\item$\Phi_{k}(h)= h(Id_E-\widetilde{A}_{1}(h)k^{-1}+O(k^{-2})),\\$ where 
\begin{align*}
 \widetilde{A}_{1}(h)=&A_{1}(h)-\bar{A_{1}},\\
 A_{1}(h)=&\frac{\sqrt{-1}}{2\pi} \Lambda F_{h}+ \frac{1}{2}S(\omega) Id_{E},\\
\bar{A_{1}}=&\frac{1}{rV} \int_{X} \tr(A_{1}(h))\frac{\omega^n}{n!} Id_E.
\end{align*}
\item$(d\Phi_{k})_{h}(\phi)=Q_{k}(\phi)+O(k^{-1}).$
\end{enumerate}
Here $\Lambda=\Lambda_\omega$ is the trace with respect to $\omega$, adjoint of Lefschetz operator. 
\end{prop}

\begin{proof}

Let $H_{k}=\Hilb(h)$ and $\{s_{1}, \dots ,s_{N_k}\}$ is an orthonormal basis for $H^{0}(X,E(k))$ with respect to $H_{k}$, i.e. $\frac{N_{k}}{rV} \int_{X}
\langle s_{i},s_{j} \rangle_{h \otimes \sigma^k}\frac{\omega^n}{n!}=\delta_{ij}.$ Thus, $$\sum s_{i}\otimes s_{i}^{*_{\FS(H_{k}) \otimes \sigma^k}}=Id_E.$$
Let $\FS({H_{k}})=h\varphi_{k}.$  Then $$\varphi^{-1}=\sum s_{i}\otimes s_{i}^{*_{h \otimes \sigma^k}}=\frac{rV}{N_{k}}B_{k}(h).$$
Hence, $$\Phi_{k}(h)=\FS(H_{k})= \frac{N_{k}}{rV}hB_{k}(h)^{-1}.$$
The second part follows from Riemann-Roch and CTYZ asymptotic expansion.

For the last part, let $h_{t}=h(Id_E+t\phi)$ and $\{s_{1}(t), \dots ,s_{N_k}(t)\}$ be an orthonormal basis for $H^{0}(X,E(k))$ with respect to $\Hilb(h_{t})$. Without loss of generality, we may assume that there exists $a_{ij}$ such that $$s_{i}(t)=s_{i}+t\sum_{j=1}^N a_{ij}s_{j}+O(t^2).$$
Let $\Phi_{k}(h_{t})=\Phi_{k}(h)\varphi_{k}(t)$. By definition, we have $$\sum s_{i}(t)\otimes s_{i}(t)^{*_{\Phi_{k}(h_{t}) \otimes \sigma^k}}=Id_{E},$$
$$\frac{N_k}{rV}\int_{X}\inner{s_{i}(t), s_{j}(t)}_{h_{t}\otimes \sigma^k} \frac{\omega^n}{n!}=\delta_{ij}.$$ 
Differentiating the second equality at $t=0$ implies $$a_{ij}+\overline{a_{ji}}=\frac{-N}{rV}\int_{X}\inner{s_{i}, \phi s_{j}}_{h\otimes \sigma^k} \frac{\omega^n}{n!}.$$
Hence, \begin{align*}  \varphi_{k}(t)^{-1}=&\sum s_{i}(t)\otimes s_{i}(t)^{*_{\Phi_{k}(h) \otimes \sigma^k}}\\=&
 \sum s_{i}\otimes s_{i}^{*_{\Phi_{k}(h) \otimes \sigma^k}}\\
&+t \Big( \sum a_{ij}s_{j}\otimes s_{i}^{*_{\Phi_{k}(h) \otimes \sigma^k}}+\sum \overline{a_{ij}}s_{i}\otimes s_{j}^{*_{\Phi_{k}(h) \otimes \sigma^k}}\Big)\\&+O(t^2)\\=&Id_E+t\frac{N_k}{rV}\widetilde{B}_{k}(h,\phi)B_{k}(h)^{-1}+O(t^2),\end{align*} where $$\widetilde{B}_{k}(h,\phi):=\frac{-N}{rV}\sum_{i,j}\Big(\int_{X}\inner{s_{i}, \phi s_{j}}_{h\otimes \sigma^k} \frac{\omega^n}{n!}\Big) s_{j}\otimes s_{i}^{*_{h\otimes \sigma^k}}.$$ Thus,
 
 \begin{align*}(d\Phi_{k})_{h}(\phi)&=\frac{d}{dt}_{\vert {t=0}} (\Phi_{k}(h)^{-1} \Phi_{k}(h_{t})) \\
&=\frac{d}{dt}_{\vert {t=0}} \varphi_{k}(t)     \\&=\frac{-N}{rV}\widetilde{B}_{k}(h,\phi)B_{k}(h)^{-1}\\
&=Q_{k}(\phi+O(k^{-1}))\\
&=Q_{k}(\phi)+O(k^{-1}).\end{align*} 
Note that Theorem \ref{Qkthm} implies that $Q_{k}(\phi+O(k^{-1}))=Q_{k}(\phi)+O(k^{-1}).$

\end{proof}

\section{Balancing Flow for bundles\label{sect3}}

Let $k \gg 0$ and $H \in \mathcal{B}_{k}$ be a hermitian inner product on $H^{0}(X,E(k))$. Let $\{s_{1}, \dots, s_{N_k}\}$ be an orthonormal basis for
$H^{0}(X,E(k))$ with respect to $H$. Using this basis, we have an embedding $\iota: X \rightarrow Gr(r,N_k)$ such that $\iota^*U_{r}=E(k)$. For any such an embedding, we can assign a moment map $\mu(\iota)\in \sqrt{-1}  \Lie{u}(N_k)$ defined by $\mu(\iota)_{ij}=\langle s_{i}, s_{j} \rangle_{\FS(H)}$. Define
$$\mu_{0}(\iota)=\mu(\iota)- \frac{rV}{N_k}Id \in   \sqrt{-1} \Lie{su}(N_k)$$ and $$\bar{\mu}_{0}(\iota)=\int_{X}\mu_{0}(\iota)(x)\frac{\omega^n_x}{n!}.$$
We also define $$\tilde{\mu}_{ij}(\iota)=s_{i}\otimes s_{j}^{*_{\FS(H)\otimes \sigma^k}}.$$

\begin{Def}
Let $H_{0} \in \mathcal{B}_{k}$ be a hermitian inner product on the vector space $H^{0}(X,E(k))$ and $\underline{s}=\{s_{1}, \dots, s_{N_k}\}$ be an orthonormal basis for
$H^{0}(X,E(k))$ with respect to $H_{0}$. Let  $\iota_{0}: X \rightarrow Gr(r,N_k)$ define by $\underline{s}$. Then the balancing flow is defined as follows:

$$\frac{d\iota(t)}{dt}=-\bar{\mu}_{0}(\iota(t)), \,\,\,\,\, \iota(0)=\iota_{0}.$$
It gives us a flow of hermitian metrics on $E$ defined by the formula $h(t)=\iota(t)^* h_{FS, G(r,N_k)}\otimes \sigma^{-k}$. Note that $h(t)$ is independent of the choice of an orthonormal basis $\{s_{1},\dots, s_{N_k}\}$.

There is a $1-1$ correspondence between the space of hermitian inner products on $H^{0}(X,E(k))$ and the space of embedding of $\iota_{0}: X \rightarrow Gr(r,N_k)$ modulo the action of $U(N_k)$. Under this correspondence, we can write the balancing flow on the space of hermitian inner products on $H^{0}(X,E(k))$, i.e. $$\frac{dH(t)}{dt}=-\bar{\mu}_{0}(H(t)), \,\,\,\,\, H(0)=H_{0}.$$

\end{Def}

The main theorem of this section is the following.

\begin{thm}
Let $h_{0}$ be a hermitian metric on $E$ and $H_{k}=\Hilb(h_{0})$. For any $k \gg 0$, let $H_{k}(t)$ be the normalized balancing flow
 $$\frac{dH_{k}(t)}{dt}=-k^{n+1}\bar{\mu}_{0}(H_{k}(t)), \,\,\,\,\, H_{k}(0)=H_{k}.$$ As before, let $h_{k}(t)=\FS(H_{k}(t)).$ Suppose that there exists a flow of metrics $h(t)$ on $E$ such that for any $T>0$, $h_{k}(t) \rightarrow h(t)$ in $C^{\infty}-$ norm for all $t \in [0,T]$. Suppose that the convergence is $C^{1}$ in $t$. Then, $h(t)$ satisfies the following PDE, \begin{align*}
h(t)^{-1}\frac{dh(t)}{dt}=&-\left( \frac{\sqrt{-1}}{2\pi}
\Lambda F_{(E,h(t))}+ \frac{1}{2}S(\omega) Id_{E}\right. \\
&\left. \hspace{0.5cm} -\frac{1}{rV}\int_{X}\tr \Big( \frac{\sqrt{-1}}{2\pi}\Lambda F_{(E,h(t))}+ \frac{1}{2}S(\omega) Id_{E}\Big)\frac{\omega^n}{n!}  Id_E\right).
\end{align*}

\end{thm}
 \begin{proof}Following Fine, we compute the change of the Fubini-Study metric as we perturb the embedding. Let $\iota: X \rightarrow G(r,N_k)$ be an embedding. Any $A \in \sqrt{-1} \Lie{u}(N)$ defines a holomorphic vector field $\xi_{A}$ on $G(r,N_k)$. This corresponds to the infinitesimal change of the embedding $\iota$ and therefore the infinitesimal change of the hermitian metric $\iota^*h_{FS}$. We claim that this change is given by $\tr(A \tilde{\mu}(\iota))$. It suffices to check this for $U_{r} \rightarrow G(r, N_k),$ where $U_{r}$ is the universal bundle over Grassmanian. Hence, the change for  the balancing flow is given by  $-\tr(\bar{\mu}_{0} \tilde{\mu}(\iota))$.
In order to prove the Theorem, it suffices to prove the following Proposition.
\end{proof}

\begin{prop}\label{prop1}
 Let $H_{k} \in \mathcal{B}_{k}$ and $h_{k}=\FS(H_{k})$. Suppose that $h_{k}$ converges in $C^{\infty}$ to a hermitian metric $h$ on $E$.
 Then $$k^{n+1}\beta_{k}(H_{k}):=-\tr(\bar{\mu}_{0}(H_{k}) \tilde{\mu}(H_{k})) \rightarrow  -A_{1}(h)+\bar{A}_{1}(h)=-\widetilde{A}_{1}(h),$$ in $C^{\infty}$.
 Here $$A_{1}(h)=\frac{\sqrt{-1}}{2\pi}\Lambda F_{(E,h(t))}+ \frac{1}{2}S(\omega) Id_{E},$$ $$\bar{A}_{1}(h)=\frac{1}{rV}\int_{X}\tr \Big(\frac{\sqrt{-1}}{2\pi}\Lambda F_{(E,h(t))}+ \frac{1}{2}S(\omega) Id_{E}\Big)\frac{\omega^n}{n!}  Id_E.$$

\end{prop}

\begin{proof}

For simplicity, suppose that there exists a hermitian metric $h$ on $E$ such that $H_{k}=\Hilb(h)$. Let $\{s_{1}, \dots, s_{N_{k}}\}$ be an orthonormal basis for $H^{0}(X,E(k))$ with respect to $H_{k}$. Let $N=N_{k}$ and $\Psi_{k}= \frac{N_{k}}{rV}B_{k}(h)^{-1}$. We have

\begin{align*}\beta_{k}(H_{k})\hspace{-0.1cm}&=\hspace{-0.1cm}-\tr(\bar{\mu}_{0}(H_{k}) \tilde{\mu}(H_{k}))\\
&=\hspace{-0.1cm}\sum_{i,j=1}^{N_k} (\int_{X}\langle s_{i}(x),s_{j}(x) \rangle_{h_{k}\otimes \sigma^k}\frac{\omega^n_x}{n!}-\frac{rV}{N_{k}}\delta_{ij})s_{j}\otimes s_{i}^{*_{h_{k}\otimes \sigma^k }}\\
&=\hspace{-0.1cm}\sum_{i,j=1}^{N_k} \int_{X}\langle s_{i}(x),s_{j}(x) \rangle_{h_{k}\otimes \sigma^k}\frac{\omega^n_x}{n!}s_{j}\otimes s_{i}^{*_{h_{k}\otimes \sigma^k }}-\frac{rV}{N_{k}}\sum_{i=1}^{N_k}s_{i}\otimes s_{i}^{*_{h_{k}\otimes \sigma^k }}\\&=\hspace{-0.1cm}\sum_{i,j=1}^{N_k} \int_{X}\langle s_{i}(x),s_{j}(x) \rangle_{h_{k}\otimes \sigma^k}\frac{\omega^n_x}{n!}s_{j}\otimes s_{i}^{*_{h_{k}\otimes \sigma^k }}-\frac{rV}{N_{k}}Id_{E}\\
&=\hspace{-0.1cm}\sum_{i,j=1}^{N_k} \int_{X}\langle s_{i}(x),\Psi_{k}(x)s_{j}(x) \rangle_{h\otimes \sigma^k}\frac{\omega^n_x}{n!}\langle ., \Psi_{k}s_{i}\rangle_{h\otimes \sigma^k} s_{j}-\frac{rV}{N_{k}}Id_{E},\end{align*}
since $h_{k}=\FS({H_{k}})=\frac{N_{k}}{rV}B_{k}(h)^{-1}h=\Psi_{k}h$. We have by doing the expansion of $N_k$ using Riemann-Roch and CTYZ expansion,
\begin{align*} \Psi_{k}&=(Id_E+\bar{A}_{1}k^{-1}+O(k^{-2}))(Id_E-A_{1}k^{-1}+O(k^{-2}))^{-1}
\\&=Id_E-\widetilde{A}_{1}k^{-1}+E_{k}, \end{align*} where $E_{k}=O(k^{-2})$. Therefore, setting $\Upsilon=Id_E-\widetilde{A}_{1}k^{-1}+E_{k}$,

 \begin{align*}\beta_{k}(H_{k})=&\sum_{i,j=1}^{N_k} \int_{X}\langle s_{i}(x),\Upsilon(x)s_{j}(x) \rangle_{h\otimes \sigma^k}\frac{\omega^n_x}{n!}\langle ., \Psi_{k}s_{i}\rangle_{h\otimes \sigma^k} s_{j}\\
&-\frac{rV}{N_{k}}Id_{E}\\
=& \Big(\sum_{i,j=1}^{N_k} \int_{X}\langle s_{i}(x),\Upsilon(x)s_{j}(x) \rangle_{h\otimes \sigma^k}\frac{\omega^n_x}{n!}\langle ., s_{i}\rangle_{h\otimes \sigma^k} s_{j}\Big)\circ \Psi_{k}\\
&-\frac{rV}{N_{k}}Id_{E}\\
=& k^{-1-n}Q_{k}(-\widetilde{A}_{1}+kE_{k})+O(k^{-2-n})\\
=&-k^{-n-1}\widetilde{A}_{1}(h)+k^{-n-1}Q_{k}(kE_{k})+O(k^{-n-2}).\end{align*} Hence, Theorem \ref{Qkthm} implies that $$\norm{Q_{k}(kE_{k})-kE_{k}}_{C^m} \leq \frac{C}{k}\norm{kE_{k}}_{C^m},$$ where $C$ is a constant independent of $k$ and $E_{k}$. Therefore, $$\norm{Q_{k}(kE_{k})}_{C^m} \leq (1+\frac{C}{k})\norm{kE_{k}}_{C^m} \to 0, \,\,\,\, k \to \infty. $$This concludes the proof.

%%%%
%Since $h_{k}=FS_{k}({H_{k}})=\frac{N_{k}}{rV}B_{k}(h)^{-1}h=\tilde{B}_{k}h$, we have \begin{align*}\beta_{k}(H_{k})&=\sum_{i,j=1}^{N_k} \int_{X}\langle s_{i},(I-A_{0}k^{-1}+\dots)s_{j} \rangle_{h\otimes \sigma^k}\frac{\omega^n}{n!}\langle ., \tilde{B}_{k}s_{i}\rangle_{h\otimes \sigma^k} s_{j}-\frac{rV}{N_{k}}I_{E}\\&= \frac{rV}{N_{k}}(\tilde{B}_{k}-I_{E})=k^{-n}\Big( \frac{rV}{N_{k}} \sum_{i=1}^{N} s_{i}\otimes  s_{i}^{B_{k}(h)^{-1}h\otimes \sigma^k}-k^{-n-1}\sum_{i,j=1}^{N} \int_{X}\langle s_{i},A_{1}(h)s_{j} \rangle_{h\otimes \sigma^k}\frac{\omega^n}{n!}s_{j}\otimes  s_{i}^{B_{k}(h)^{-1}h\otimes \sigma^k}\Big)\\&=k^{-n}\Big( (\frac{rV}{N_{k}})^2 \sum_{i=1}^{N} s_{i}\otimes  s_{i}^{h_{k}\otimes \sigma^k}-k^{-2n-1}\sum_{i,j=1}^{N} \int_{X}\langle s_{i},A_{1}(h)s_{j} \rangle_{h\otimes \sigma^k}\frac{\omega^n}{n!}s_{j}\otimes  s_{i}^{h\otimes \sigma^k}+O(k^{-2n-2})\Big)\\&=k^{-n}\Big( (\frac{rV}{N_{k}})^2 I_{E}-k^{-2n-1}Q_{k}(A_{1}(h))+O(k^{-2n-2})\Big).\end{align*} By Riemann-Roch, we have $N_{k}=rk^n(1+\bar{A_{1}}k^{-1}+O(k^{-2}))$. Hence $$\beta_{k}(H_{k})=k^{-n-1}(\bar{A_{1}}-Q_{k}(A_{1}(h))+O(k^{-1})).$$

%%%%
\end{proof}

\section{Hermitian-Einstein flow to balancing flow \label{sect4}}

Let $h$ be a hermitian metric on $E$. Then there exists a flow of metrics on $E$ as follows:

\begin{align}\label{almostDonHeatflow}
\left\{\begin{array}{ll}
h(t)^{-1}\frac{dh(t)}{dt}&=-\left( \frac{\sqrt{-1}}{2\pi} \Lambda F_{(E,h(t))}+ \frac{1}{2}S(\omega) Id_{E}\right.\\
&\hspace{1cm}\left. -\frac{1}{rV}\int_{X}\tr \Big( \frac{\sqrt{-1}}{2\pi} \Lambda F_{(E,h(t))}+ \frac{1}{2} S(\omega) Id_{E}\Big)\frac{\omega^n}{n!}\right) Id_E,\\
h(0)&=h.
\end{array}\right.
\end{align}
\begin{Rmk}\label{rmk1}
Note that, using a conformal change of the metric (see Corollary \ref{cor1}), this flow can be obtain from Donaldson heat flow for hermitian endomorphisms of the bundle, namely
$$h(t)^{-1}\frac{dh(t)}{dt}=-\left( \frac{\sqrt{-1}}{2\pi} \Lambda F_{(E,h(t))} -\mu(E) Id_E\right ).$$
Here $\mu(E)$ is the slope of $E$ with respect to the polarization on $X$. It is well known that a unique smooth solution of this flow exists for all $t\in [0,+\infty)$.\\
Moreover, this flow is equivalent by a change of the holomorphic structure to the classical Yang-Mills flow for connections. Let us give some details. The initial metric $h(0)$ provides in a unitary frame a connection $A_0$ that decomposes as $(1,0)$ and $(0,1)$ components that we shall denote $A_0'$ and $A_0''$ respectively. Starting the initial holomorphic structure $\bar{\partial}_0=\bar{\partial}_E+A_0'',$ we obtain a flow of holomorphic structures 
$$\bar{\partial}_t = \bar{\partial}_E+ A''(t), $$
by setting $A''(t)=h(t)^{1/2}A_0''h(t)^{-1/2} - \bar{\partial}h(t)^{1/2}\,h(t)^{-1/2}$. Fixing the initial metric, this flow of holomorphic structures induces a flow of unitary connections $A(t)=A'(t)+A''(t)$ on the bundle. The induced curvature, i.e $F_{A(t)}=\bar\partial A'(t)+\partial A''(t) + A'(t)\wedge A''(t) + A''(t)\wedge A'(t)$, satisfies the Yang-Mills flow equation
$$\frac{d}{dt}A(t)= -d^{*}_{A(t)} F_{A(t)}=\sqrt{-1}\bar{\partial}_t \Lambda_\omega F_{A(t)} - \sqrt{-1} {\partial}_t \Lambda F_{A(t)}.$$
Here the last equality is a consequence of the K\"ahler identities and Bianchi's identity. Note that the Yang-Mills flow preserves the integrability condition $\bar{\partial}^2_{t}=0$, and so $A(t)$ defines a holomorphic structure for all $t$. We refer to \cites{Do85,Do87} for details.
\end{Rmk}

At that stage we need a general simple result.
\begin{prop}\label{deltaH_k}

Let $h_{t}$ be a smooth family of hermitian metrics on $E$ such that $h_{0}=h$. Let $\widehat{H}_{k}(t)=\Hilb(h_{t})$ and $\{s_{1}, \dots ,s_{N_k}\}$ be an orthonormal basis for $H^{0}(X,E(k))$ with respect to $\widehat{H}_{k}(t)$. Then the infinitesimal change $\delta H_{k}=\widehat{H}_{k}(t)^{-1}\frac{d \widehat{H}_{k}(t)}{dt}$ is given as follows:
$$(\delta H_{k})_{ij}=-\frac{N_{k}}{rV} \int_{X}\langle s_{i},\dot{\phi_{t}}s_{j} \rangle_{h_{t}\otimes \sigma^k}\frac{\omega^n_x}{n!}, $$ where $\dot{\phi_{t}}=h(t)^{-1}\frac{d h(t)}{d t}$. Moreover,
$$\frac{\tr((\delta H_{k})^2)}{k^n}=\int_{X}\tr(Q_{k}(\dot{\phi})\dot{\phi})\frac{\omega^n_x}{n!} \rightarrow \int_{X}\tr(\dot{\phi}^2)\frac{\omega^n_x}{n!}, \,\text{ as } k \rightarrow +\infty. $$

\end{prop}

\begin{proof}

Let $\{s_{1}(t), \dots ,s_{N_k}(t)\}$ be an orthonormal basis for $H^{0}(X,E(k))$ with respect to  $\widehat{H}_{k}(t)$. Therefore,
$$\frac{N_{k}}{rV} \int_{X}\langle s_{i}(t),s_{j}(t) \rangle_{h_{t} \otimes \sigma^k}\frac{\omega^n}{n!}=\delta_{ij}.$$ Differentiating with respect to $t$, we have
$$(\delta H_{k})_{ij}+\frac{N_{k}}{rV} \int_{X}\langle s_{i},\dot{\phi_{t}}s_{j} \rangle_{h \otimes \sigma^k}\frac{\omega^n}{n!}=0.$$

%By definition, $\delta H_{k}=\frac{d}{dt}|_{t=0}H_{k}(t) H_{k}^{-1}.$ Therefore, $$(\delta H_{k})_{ij}=-\frac{d}{dt}|_{t=0}\frac{N_{k}}{rV} \int_{X}\langle s_{i}(t),s_{j}(t) \rangle_{h_{t} \otimes \sigma^k}\frac{\omega^n}{n!}$$
%$$-\frac{d}{dt}|_{t=0}\frac{N_{k}}{rV} \int_{X}\langle s_{i}(t),e^{\phi_{t}}s_{j}(t) \rangle_{h \otimes \sigma^k}\frac{\omega^n}{n!}=$$

%$$=\frac{N_{k}}{rV} \int_{X}\langle s_{i},\dot{\phi_{t}}s_{j} \rangle_{h \otimes \sigma^k}\frac{\omega^n}{n!}.$$

\end{proof}

From now on $h(t)$ satisfies the above, $\widehat{H}_{k}(t)=\Hilb (h(t))$  and $H_{k}(t)$ is the balancing flow starting at $\widehat{H}_{k}(0)$.

\begin{lem}\label{lemma1}

Let $\widehat{H}_{k}(t)=\Hilb(h(t))$ and $U_{k}(t)$ be the tangent vector to $\widehat{H}_{k}(t)$. Let $V_{k}(t)$ be the tangent vector to the normalized balancing flow started at $\widehat{H}_{k}(t)$. Then $$\frac{\tr(U_{k}(t)-V_{k}(t))^2}{k^n}=O(k^{-2}).$$

\end{lem}

\begin{proof}
Let $\{s_{1}, \dots ,s_{N_k}\}$ be an orthonormal basis for $H^0(X,E(k))$ with respect to $\widehat{H}_{k}(t).$ We have from  Proposition \ref{deltaH_k},
\begin{align*} U_{k}(t)=-k^n\int_{X}\inner{s_{i},\widetilde{A}_{1}(h)s_{j}}_{h}\frac{\omega^n}{n!}+O(k^{n-1})\end{align*}
Moreover,
\begin{align*} V_{k}(t)&=-k^{n+1}\bar{\mu}_{0}(\widehat{H}_{k}(t))\\
&=k^{n+1} \int_{X}\Big(\frac{rV\delta_{ij}}{N_{k}}-
\inner{s_{i},s_{j}}_{\FS(\widehat{H}_{k}(t))}\Big)\frac{\omega^n}{n!} \\&=k^{n+1}\int_{X}\Big(\frac{rV\delta_{ij}}{N_{k}}-\inner{s_{i},\widetilde{B}_{k}(h(t))s_{j}}_{h(t)\otimes \sigma^k}\Big)\frac{\omega^n}{n!} \\&=k^{n+1}\Big(  \frac{rV\delta_{ij}}{N_{k}}-\int_{X} \inner{s_{i},s_{j}}\frac{\omega^n}{n!} -\frac{1}{k}\int_{X} \inner{s_{i},\widetilde{A}_{1}(h)s_{j}}\frac{\omega^n}{n!} +O(k^{-2})\Big)\\&=-k^n\int_{X} \inner{s_{i},\widetilde{A}_{1}(h)s_{j}}\frac{\omega^n}{n!} +O(k^{n-1}).\end{align*}
Therefore, there exists $\epsilon_{k}=O(k^{-1})$ such that

$$U_{k}(t)-V_{k}(t)=k^{n}\int_{X} \inner{s_{i},\epsilon_{k}s_{j}}\frac{\omega^n}{n!}.$$
We have, 

\begin{align*}
 k^{-n}\tr  (U_{k}-V_{k})^2 =&k^n\int_{X \times X}\sum_{i,j}\inner{s_{i}(x),\epsilon_{k}(x)s_{j}(x)} \inner{s_{j}(y),\epsilon_{k}(y)s_{i}(y)}\\
=&\int_{X} \tr(Q_k(\epsilon_{k})\epsilon_{k})=O(k^{-2}). \end{align*}
The last equality follows from the fact that $Q_k(\epsilon_{k})=\epsilon_{k}+O(k^{-2})$. 
\end{proof}

\subsection{First Order Approximation\label{deform1}}

We start this section by the following definition.
\begin{Def}
Using the Killing form on the space of hermitian metrics on $H^{0}(X,E(k))$, we defined the normalized distance $d_{k}$ on the space of hermitian metrics on $H^{0}(X,E(k))$ by $\sqrt{k^{-n}\tr(A^2)}$.
\end{Def}

We have the following Proposition.

\begin{prop}
Given $T >0$, there exists a constant $C=C(T)$ such that $$d_{k}(H_{k}(t),\widehat{H}_{k}(t)) \leq \frac{C}{k},\,\,\,\, \text{for all }  \, t\in [0,T].$$

\end{prop}

\begin{proof}
As before, let $U_{k}(t)$ be the tangent vector to $\widehat{H}_{k}(t)$ and $V_{k}(t)$ be the tangent to the normalized balancing flow started at $\widehat{H}_{k}(t)$.
By Lemma \ref{lemma1}, we know that $\norm{U_{k}-V_{k}}_{d_{k}}=O(k^{-1})$ uniformly in $t$. Let $\widetilde{H}_{k}(t)$ be the balancing flow started at $t=t_{0}$ by $\widehat{H}_{k}(t_{0})$ (i.e. $\widetilde{H}_{k}(t)|_{t=t_{0}}=\widehat{H}_{k}(t_{0})$). Then $\widehat{H}_{k}(t)$ and $\widetilde{H}_{k}(t)$ are tangent at $t=t_{0}$ up to error of size $O(k^{-1}).$
Define $$f_{k}(t)= d_{k}(H_{k}(t),\widehat{H}_{k}(t)),$$
$$\widetilde{f}_{k}(t)= d_{k}(\widetilde{H}_{k}(t),\widehat{H}_{k}(t)).$$ 
Since $\widetilde{H}_{k}(t_{0})=\widehat{H}_{k}(t_{0})$, we have \begin{align*} f_{k}(t)-f_{k}(t_{0})&=d_{k}(H_{k}(t),\widehat{H}_{k}(t))-d_{k}(H_{k}(t_{0}),\widehat{H}_{k}(t_{0}))\\&\leq  \widetilde{f}_{k}(t)-\widetilde{f}_{k}(t_{0})+d_{k}(\widetilde{H}_{k}(t),H_{k}(t))-d_{k}(\widetilde{H}_{k}(t_{0}),H_{k}(t_{0})).\end{align*}
Therefore, \begin{align*}\frac{d}{dt}_{\vert {t=t_0}}f_{k}(t)&\leq \frac{d}{dt}_{\vert {t=t_0^{+}}}\widetilde{f}_{k}(t)+\lim_{t \to t_{0}^+}\frac{d_{k}(H_{k}(t),\widetilde{H}_{k}(t))-d_{k}(H_{k}(t_{0}),\widetilde{H}_{k}(t_{0}))}{t-t_{0}}\\&=O(k^{-1})+\lim_{t \to t_{0}^+}\frac{d_{k}(H_{k}(t),\widetilde{H}_{k}(t))-d_{k}(H_{k}(t_{0}),\widetilde{H}_{k}(t_{0}))}{t-t_{0}}.\end{align*} On the other hand $\frac{d}{dt}( d_{k}(H_{k}(t),\widetilde{H}_{k}(t)) )\leq 0$, since the balancing flow is a downward gradient flow and thus distance decreasing. Therefore, $$\frac{d}{dt}_{\vert {t=t_0}} f_{k}(t)\leq Ck^{-1}$$ uniformly with respect to $k$ and $t_{0}\in [0,T]$. Since $f_{k}(0)=0$, then $$f_{k}(t) \leq \frac{CT}{k},$$ for all $t \in [0,T]$.

\color{black}

\end{proof}

\subsection{Higher Order Approximation\label{deform2}}

\begin{thm}\label{higher}
There exist $\phi_{1}(t), \phi_{2}(t), \dots \in \Gamma(X, End(E))$, hermitian with respect to $h(t)$, solution of \eqref{almostDonHeatflow}, such that for any $q \in \mathbb{N}$, we have
$$d_{k}(H(k;t), \widehat{H}(k;t)) \leq \frac{C}{k^{q+1}}\,\,\,\, \text{for all }  \, t\in [0,T], $$ where $C=C(T)$. Here,
$$h(k;t)=h(t)\left(Id_E+\sum_{j=1}^{q}k^{-j}\phi_{j}(t)\right),$$
$\widehat{H}(k;t)=\Hilb(h(k;t))$ and $H(k;t)$ is the balancing flow started at the metric $\Hilb(h(k;0))$.
\end{thm}

\begin{proof}
First, we explain how to construct $\phi_{1}(t)$. Define $$h(k;t)=h(t)(Id_E+k^{-1}\phi_{1}(t)),$$ and $\widehat{H}(k;t)=\Hilb(h(k;t)).$ Let $U_{k}(t)$ and $V_{k}(t)$ be the tangent vectors to $\widehat{H}(k;t)$ and the balancing flow starting $\widehat{H}(k;t)$ respectively. Then we have that $$\Gamma_t:= h(k;t)^{-1}\frac{d}{dt}h(k;t)$$ is equal to 
\begin{align*} \Gamma_t=&(Id_E+k^{-1}\phi_{1}(t))^{-1}h(t)^{-1}\times \\
&\hspace{1cm}\Big(  k^{-1} h(t)\frac{d \phi_{1}(t)}{dt} +\frac{d h(t)}{dt}(Id_E+k^{-1}\phi_{1}(t))  \Big)\\=& k^{-1}(Id_E+k^{-1}\phi_{1}(t))^{-1}\frac{d \phi_{1}(t)}{dt} \\
&+(Id_E+k^{-1}\phi_{1}(t))^{-1}h(t)^{-1}\frac{d h(t)}{dt}(Id_E+k^{-1}\phi_{1}(t)) \\=&h(t)^{-1}\frac{d h(t)}{dt}+k^{-1} \Big(\frac{d \phi_{1}(t)}{dt}-\Big[\phi_{1}(t),h(t)^{-1}\frac{d h(t)}{dt}\Big]\Big) +O(k^{-2})\\=&-\widetilde{A}_{1}(h)+k^{-1} \Big(\frac{d \phi_{1}(t)}{dt}+[\phi_{1}(t),\widetilde{A}_{1}(h)]\Big) +O(k^{-2}).\end{align*}
Now, let $\{s_{1}, \dots ,s_{N_k}\}$ be an orthonormal basis  with respect to the metric $\widehat{H}(k;t)=\Hilb(h(k;t)).$ Therefore, Proposition \ref{deltaH_k} implies that
\begin{align*}U_{k}(t)_{ij}=&-\frac{N_{k}}{rV} \int_{X}\inner{s_{i}, \Big(h(k;t)^{-1}\frac{d}{dt}h(k;t)\Big) s_{j}}\\=&\frac{N_{k}}{rV} \int_{X}\inner{s_{i}, \widetilde{A}_{1}(h)s_{j}}\\
&-\frac{N_{k}}{rV}\left(\frac{1}{k} \int_{X}\inner{s_{i},\Big(\frac{d \phi_{1}(t)}{dt}+[\phi_{1}(t),\widetilde{A}_{1}(h)]\Big)   s_{j}}+ O(\frac{1}{k^2})\right)\end{align*} Hence,
\begin{align*}
k^{-n}U_{k}(t)_{ij}=&\int_{X}\inner{s_{i}, \widetilde{A}_{1}(h) s_{j}}\\
&+\frac{1}{k}\Big(\bar{A}_{1}\int_{X}\inner{s_{i}, \widetilde{A}_{1}(h) s_{j}}-\int_{X}\inner{s_{i},\Big(\frac{d \phi_{1}}{dt}+[\phi_{1},\widetilde{A}_{1}(h)]\Big)   s_{j}}\Big)\\
&+O(k^{-2}).
\end{align*}

On the other hand \begin{align*} k^{-n}V_{k}(t)&=k^{-n}\frac{d}{dt}\widetilde{H}(k;t)\\
&=-k\bar{\mu}_{0}(\widetilde{H}_{k}(t))\\
&=-k \int_{X}\inner{s_{i},s_{j}}_{\FS(\widehat{H}_{k}(t))}+\frac{kr\delta_{ij}}{N_k}\\&= -k \int_{X}\inner{s_{i},s_{j}}_{\Phi_{k}(h(k;t))}+\frac{kr\delta_{ij}}{N_k}.\end{align*}
We have from Proposition \ref{prop1a},  $$\Phi_{k}(h(k;t))=Id_E-\widetilde{A}_{1}(h(k;t))k^{-1}+\Big({A_{1}^2}-\bar{A}_{1}A_{1}-A_{2}+\bar{A}_{2}\Big)k^{-2}+O(k^{-3}),$$
where $B_{k}(h(t))=k^n+A_{1}k^{n-1}+A_{2}k^{n-2}+O(k^{n-3})$ and $$\bar{A}_{i}= \frac{1}{rV}\int_{X} \tr(A_{i}) Id_E.$$ Since $\widetilde{A}_{1}(h(k;t))=\widetilde{A}_{1}(h(k;t))+k^{-1}A_{1,1}(h(t))\phi_{1}(t)$, we have
\begin{align*}
\Phi_{k}(h(k;t))=&Id_E-\widetilde{A}_{1}k^{-1}\\
&\hspace{-0.1cm}+({A_{1}^2}-\bar{A}_{1}A_{1}-A_{2}+\bar{A}_{2}-A_{1,1}(h(t))\phi_{1}(t))k^{-2}+O(k^{-3}). 
\end{align*}
Here $A_{1,1}(h(t))\phi_{1}(t) $ is the linearization of $A_{1}$ at $h(t)$. It is easy to see that $$A_{1,1}(h(t))\phi_{1}(t)=\sqrt{-1}\Lambda \bar\partial\partial_t \phi_1(t)=\Delta_{t}\phi_{1}(t).$$ Here
$\Delta_{t}$ is the $\bar{\partial}$-laplacian on the space of sections of $End(E)$ with respect to the metric $h(t)$. Let $$\Sigma={A_{1}^2}-\bar{A}_{1}A_{1}-A_{2}+\bar{A}_{2}$$ and $$A_{1,1}=A_{1,1}(h(t))\phi_{1}(t)= \Delta_{t}\phi_{1}(t).$$ We have
\begin{align*} k^{-n}V_{k}(t)&=-k \int_{X}\inner{s_{i},s_{j}}_{\Phi_{k}(h(k;t))}+\frac{kr\delta_{ij}}{N_k}\\&=\frac{kr\delta_{ij}}{N_k}-k\int_{X}\inner{s_{i},  s_{j}}+\int_{X}\inner{s_{i}, \widetilde{A}_{1}(h) s_{j}}\\&-k^{-1}\int_{X}\inner{s_{i}, \Sigma s_{j}} +k^{-1}\int_{X}\inner{s_{i}, A_{1,1} s_{j}}+O(k^{-2})\\&=\int_{X}\inner{s_{i}, \widetilde{A}_{1}(h) s_{j}}-k^{-1}\Big(\int_{X}\inner{s_{i}, \Sigma s_{j}} - \int_{X}\inner{s_{i}, A_{1,1} s_{j}}\Big)\\
&+O({k^{-2}}).\end{align*} Hence,

\begin{align*}\Big( \frac{ U_{k}(t)-V_{k}(t)}{k^{n+1}} \Big)_{ij}=&\int_{X}\inner{s_{i}, \Sigma s_{j}}-\int_{X}\inner{s_{i}, A_{1,1} s_{j}}+\bar{A}_{1}\int_{X}\inner{s_{i}, \widetilde{A}_{1}(h) s_{j}}
\\&-\int_{X}\inner{s_{i},\Big(\frac{d \phi_{1}(t)}{dt}+[\phi_{1}(t),\widetilde{A}_{1}(h)]\Big)   s_{j}}+O(k^{-1}).\end{align*}
%\begin{align*}k^{-n}B_{k}(t)_{ij}&=k\int_{X}\Big( \frac{r\delta_{ij}}{N_{k}}-\inner{s_{i},s_{j}}_{\FS{\widehat{H}_{k}(t)}}  \Big)=k\int_{X}\Big( \frac{r\delta_{ij}}{N_{k}}-\inner{s_{i},B_{k}(h(k;t))^{-1}s_{j}}_{h(k;t)} } \Big)\\&=\int_{X}\inner{s_{i},\widetilde{A}_{1}s_{j}}+k^{-1}\int_{X} \inner{???-A_{1,1}(\phi_{1})s_{j}}+O(k^{-2}).\end{align*} 
We remark that $[\phi,\widetilde{A}_{1}(h)]=-[\frac{\sqrt{-1}}{2\pi}\Lambda F_{(E,h(t))},\phi]$. Now, suppose that $\phi_{1}(t) \in \Gamma(X,End(E))$ solves the following second-order PDE \begin{equation}\frac{d \phi_{1}(t)}{dt}+\Delta_t\phi_1 - \Big[\frac{\sqrt{-1}}{2\pi} \Lambda F_{(E,h(t))},\phi_1\Big]=\Sigma+\bar{A}_{1}\widetilde{A}_{1}, \label{parab-eqn}\end{equation} with initial condition $\phi_{1}(0)=0$. Suppose also that $\phi_1(t)$ is self-adjoint with respect to $h(t)$ and thus $h(k,t)$ is well defined. Then $k^{-n-1}(U_{k}(t)-V_{k}(t)) =O(k^{-2})$, and therefore we can show that $$d_{k}(H(k;t), \widehat{H}(k;t)) \leq \frac{C}{k^{2}}\,\,\,\, \text{for all }  \, t\in [0,T], $$ where $C=C(T),$ as expected. \\
Equation \eqref{parab-eqn} is a linear parabolic equation and admits a solution $\phi_1$ for all time from general theory on linear parabolic systems with smooth data on compact manifolds. We refer to the following references for details on such systems \cite[Chapter VI, Theorem 5.8]{Kob}, \cite[Section 3.1]{Jost}, \cite[Sections 13.9 \& 15]{Fee}, \cite[Chapter VII]{LSU}. The solution $\phi_1$ is unique when one assumes the initial condition $\phi_1(0)=0$, as a consequence of maximum principle.\\
\label{adjoint-reasoning}Eventually, we need to check that we constructed a self-adjoint operator $\phi_1$ when solving \eqref{parab-eqn}. First we remark that
\begin{equation}\label{eqpara3} \Sigma+\bar{A}_{1}\widetilde{A}_{1}= 2A_1^2- 2A_1\bar A_1 +\bar A_2 - A_2 = \left( \Sigma+\bar{A}_{1}\widetilde{A}_{1} \right)^*,
\end{equation}
as a consequence of  Proposition \ref{prop1a} (1) and the fact that the terms of the asymptotics of the Bergman kernel are self-adjoint. Furthermore, taking the adjoint with respect to $h(t)$, we obtain by K\"ahler identities,
\begin{align}
\left(\Delta_t \phi_1 - \Big[\frac{\sqrt{-1}}{2\pi} \Lambda F_{(E,h(t))},\phi_1\Big]\right)^* =& -(\sqrt{-1}\Lambda \partial\bar\partial \phi_1)^* \nonumber\\
=&\sqrt{-1}\Lambda (\partial \bar\partial \phi_1)^*\nonumber \\
=&\sqrt{-1}\Lambda \bar\partial (\overline{\partial\phi_1})^*\nonumber\\
=&\sqrt{-1}\Lambda \bar\partial\partial \phi_1^*.  \label{eqpara2}
\end{align}
Moreover, for any vector $X,Y\in E$, 
$$\langle \frac{d \phi_{1}(t)}{dt}X,Y\rangle_{h(t)} + \langle \phi_1(t)X,Y\rangle_{\frac{dh(t)}{dt}} = \langle X, \frac{d \phi_{1}^*(t)}{dt} Y\rangle_{h(t)}+ \langle X,\phi_1^*(t) Y\rangle_{\frac{dh(t)}{dt}}.$$
Thus, using \eqref{almostDonHeatflow}, we get 
\begin{align} 
 \left(\frac{d \phi_{1}(t)}{dt}\right)^* =& \frac{d \phi_{1}^*(t)}{dt} - \frac{\sqrt{-1}}{2\pi}\Lambda F_{(E,h(t))} \phi_1^*(t) + \phi_1^* \frac{\sqrt{-1}}{2\pi}\Lambda F_{(E,h(t))}\nonumber \\
=&  \frac{d \phi_{1}^*(t)}{dt} -\Big[\frac{\sqrt{-1}}{2\pi} \Lambda F_{(E,h(t))},\phi_1^*\Big] \label{eqpara1}
\end{align}
Now, with \eqref{eqpara1}, \eqref{eqpara2} and \eqref{eqpara3}, $\phi_1^*$ is solution of \eqref{parab-eqn} as soon as $\phi_1$ is solution of \eqref{parab-eqn}. By uniqueness, it implies that $\phi_1$ is self adjoint. \\ 

Now by induction on $q$, we can prove the Theorem.\\

%%%%%%%%%%%%%%%%%
\textbf{Induction on $q$:}
%%%%%%%%%%%%%%%%%

Suppose that the endomorphisms $\phi_{1}(t), \dots ,\phi_{q-1}(t)$ are chosen such that $U_{k}(t)-V_{k}(t) =O(k^{-q+n+1})$ and $\phi_{1}(0)=\dots=\phi_{q-1}(0)=0$.
Define $$h(k;t)=h(t)(Id_E+\sum_{j=1}^{q}k^{-j}\phi_{j}(t)),$$
$\widehat{H}(k;t)=\Hilb(h(k;t))$. We need to find $\phi_{q}(t)$ such that $U_{k}(t)-V_{k}(t) =O(k^{-q-1+n+1})$.
We have, defining $\Gamma_{k,t}:= h(k;t)^{-1}\frac{d}{dt}h(k;t)$, 
\begin{align*}  \Gamma_{k,t}=&\big(Id_E+\sum_{i=1}^{q}k^{-i}\phi_{i}(t)\big)^{-1}h(t)^{-1}\times \\
 &\hspace{1cm}\Big(   h(t)\frac{d }{dt}\big( Id_E+\sum_{j=1}^{q}k^{-j}\phi_{j}(t)\big) +\frac{d h(t)}{dt}(Id_E+\sum_{j=1}^{q}k^{-j}\phi_{j}(t))  \Big)\\=&\big(Id_E+\sum_{i=1}^{q}k^{-i}\phi_{i}(t)\big)^{-1}h(t)^{-1}\times \\
& \hspace{1cm}\Big(   h(t)\big( \sum_{j=1}^{q}k^{-j}\frac{d }{dt}\phi_{j}(t)\big) +\frac{d h(t)}{dt}(Id_E+\sum_{j=1}^{q}k^{-j}\phi_{j}(t))  \Big)\\=&\big(Id_E+\sum_{i=1}^{q}k^{-i}\phi_{i}(t)\big)^{-1}\big( \sum_{j=1}^{q}k^{-j}\frac{d }{dt}\phi_{j}(t)\big) \\&\,\,\,\,\,\,+\big(Id_E+\sum_{i=1}^{q}k^{-i}\phi_{i}(t)\big)^{-1}h(t)^{-1}\frac{d h(t)}{dt}\big(Id_E+\sum_{j=1}^{q}k^{-j}\phi_{j}(t)\big)  \\=&-\widetilde{A}_{1}(h)+\dots +k^{-q} \Big(\frac{d \phi_{q}(t)}{dt}+[\phi_{q}(t),\widetilde{A}_{1}(h)]+R\Big)\\
& +O(k^{-q-1}),\end{align*} where $R=R(h(t),\phi_{1}(t),\dots ,\phi_{q-1}(t))$  is smooth. $R$ is also self-adjoint since the $(\phi_{j})_{j=1,..,q-1}$ are self-adjoint by induction process. Therefore,

\begin{align*}U_{k}(t)_{ij}=&-\frac{N_{k}}{rV} \int_{X}\inner{s_{i}, \Big(h(k;t)^{-1}\frac{d}{dt}h(k;t)\Big) s_{j}},\\=&\frac{N_{k}}{rV} \int_{X}\inner{s_{i}, \widetilde{A}_{1}(h)s_{j}}+\dots\\
&-k^{-q} \int_{X}\inner{s_{i},\Big(\frac{d \phi_{q}(t)}{dt}+[\phi_{q}(t),\widetilde{A}_{1}(h)]+R\Big)    s_{j}}\\
&+ O(k^{-q-1}).\end{align*} 
On the other hand, we have 
\begin{align*}\Phi_{k}(h(k;t))=&Id_E-\widetilde{A}_{1}(h(t))\frac{1}{k}+\dots +\left(\Sigma - A_{1,1}(h(t))\phi_{q}(t)\right) \frac{1}{k^{q+1}}\\
 &+O(k^{-q-2}),
\end{align*}
where
$\Sigma=\Sigma(h(t),\phi_{1}(t),\dots, \phi_{q-1}(t)).$ Hence, 
\begin{align*} k^{-n}V_{k}(t)=&k^{-n}\frac{d}{dt}\widetilde{H}(k;t)\\
=&-k\bar{\mu}_{0}(\widetilde{H}_{k}(t))\\
=&-k \int_{X}\inner{s_{i},s_{j}}_{\FS(\widehat{H}_{k}(t))}+\frac{kr\delta_{ij}}{N_k}\\=&k \int_{X}\inner{s_{i},s_{j}}_{\Phi_{k}(h(k;t))}+\frac{kr\delta_{ij}}{N_k}\\=&\int_{X}\inner{s_{i}, \widetilde{A}_{1}(h) s_{j}}+\dots \\
&-k^{-q}\Big(\int_{X}\inner{s_{i}, \Sigma s_{j}} -\int_{X}\inner{s_{i}, A_{1,1}(\phi_{q}) s_{j}}\Big)\\
&+O(k^{-q-1}).\end{align*}
The induction assumption implies that
\begin{align*}\Big( \frac{U_{k}(t)-V_{k}(t)}{k^{n+1}}  \Big)_{ij}=&\frac{1}{k^{q}}\Big(\int_{X}\inner{s_{i}, (\Sigma-A_{1,1}(\phi_{q})) s_{j}}\\
&\hspace{0.7cm}-\int_{X}\inner{s_{i},\Big(\frac{d \phi_{q}(t)}{dt}+[\phi_{q}(t),\widetilde{A}_{1}(h)]+R\Big)    s_{j}}\Big)\\&+ O(k^{-q-1})\\=&\frac{-1}{k^{q}}\int_{X}\inner{s_{i},\Big(\frac{d \phi_{q}}{dt}+[\phi_{q},\widetilde{A}_{1}(h)]+ \tilde{R}+\Delta_{t}\phi_{q}\Big)    s_{j}}\\
&+ O(k^{-q-1}).\end{align*}
Here $\tilde{R}=R-\Sigma$ is self-adjoint since each term of the asymptotic expansion of $\Phi_{k}(h(k;t))$ is self-adjoint.
As we have seen before, there exists a unique smooth solution  $\phi_{q}(t) \in \Gamma(X,End(E))$ of the linear parabolic PDE \begin{equation}
\frac{d \phi_{q}(t)}{dt}+\Delta_t\phi_q(t) - \Big[\frac{\sqrt{-1}}{2\pi}F_{(E,h(t))},\phi_q\Big] +\tilde{R}= 0,\end{equation} with initial condition $\phi_{q}(0)=0$. This solution provides a hermitian metric $h(k;t)$ by the same reasoning as at page \pageref{adjoint-reasoning}. For such a solution $h(k;t)=h(t)(Id_E+\sum_{j=1}^{q}k^{-j}\phi_{j}(t))$ one has $$k^{-n-1}(U_{k}(t)-V_{k}(t)) =O(k^{-q-1}),$$ and therefore we can show that $$d_{k}(H(k;t), \widehat{H}(k;t)) \leq \frac{C}{k^{q+1}}\,\,\,\, \text{for all }  \, t\in [0,T], $$ where $C=C(T).$ This concludes the proof of Theorem \ref{higher}.
\end{proof}

\section{Final estimates\label{estimates}}
We fix a metric $h_{0}$ on $E$. Let $\underline{s}=\{s_{1}, \dots ,s_{N_k}\}$ be an orthonormal basis for $H^{0}(X,E(k))$ with respect to the normalized inner product $\Hilb(h_{0})$. Using $\underline{s}$, we have an embedding $\iota_{\underline{s}}:X \to Gr(r,N_{k})$. Define the metric $\widetilde{h}_{0}$ on $E(k)$ by $\widetilde{h}_{0}=\iota_{\underline{s}}^*h_{FS}$. Note that $\widetilde{h}_{0}\sim h_{0}\otimes \sigma^k.$ Through this section, all norms and inner products are with respect to $\widetilde{h}_{0}$ unless it is specified otherwise.

\begin{Def}
Fix $R >0$ and $m \geq 4$ a positive integer. We say that a metric $h$ on $E(k)$ has $R$-bounded geometry if
\begin{enumerate}
\item $\displaystyle h \geq \frac{1}{R}\widetilde{h}_{0},$
\item $\displaystyle \norm{h -\widetilde{h}_{0}}_{C^m(\widetilde{h}_{0})} \leq R.$
\end{enumerate}
Note that Condition (2) implies that $h \leq (R+1)\widetilde{h}_{0}.$ \\
We say a basis $\underline{s}=\{s_{1}, \dots ,s_{N_k}\}$ for $H^{0}(X,E(k))$ has $R$-bounded geometry if the Fubini-Study metric induced by $\underline{s}$ on $E(k)$ has $R$-bounded geometry.
\end{Def}

\begin{Def}

Let  $\underline{s}=\{s_{1}, \dots ,s_{N_k}\}$ be a basis for $H^{0}(X,E(k))$ and $A=(a_{ij}) \in \sqrt{-1}\Lie{su}(N_k)$. There exists a unique metric $h$ on $E(k)$
such that $$\sum s_{\alpha} \otimes s_{\alpha}^{*_{h}}=Id_{E}.$$
We define $$H_{A}=\sum a_{\alpha \beta}s_{\alpha} \otimes s_{\beta}^{*_{h}}.$$
Note $tr(H_{A})$ is the hamiltonian associated to the vector field $\zeta_{A}$ defined by \eqref{zeta}.
\end{Def}

We have the following Lemma.

\begin{lem}

Let  $\underline{s}=\{s_{1}, \dots ,s_{N_k}\}$ be a basis for $H^{0}(X,E(k))$ and  $\iota: X \rightarrow Gr(r,N_k)$ be the corresponding embedding. 
Suppose that $\iota_{t}: X \rightarrow Gr(r,N_k)$ be a smooth family of embeddings such that $\iota_{0}=\iota$ and $\frac{d}{dt}_{\vert {t=0}} \iota_{t}=A$.
(i.e. suppose $\iota_{t} $ is induced by a family of bases $\underline{s}(t)=\{s_{1}(t), \dots ,s_{N_k}(t)\}$ for $H^{0}(X,E(k))$ such that $\underline{s}(0)=\underline{s}.$ Then $A:=\frac{d}{dt}_{\vert {t=0}} \underline{s}(t)$).\\
Let $h_{t}$ be the induced Fubini-Study metrics on $E(k)$ using $\underline{s}(t)$.
Then $h_{t}=e^{\phi_{t}}h$ such that $$\frac{d}{dt}_{\vert {t=0}}\phi_{t}=H_{A}.$$
\end{lem}
The proof is straightforward.\\

We need the following Proposition, by a slight modification of an argument from \cite{W2}. 

\begin{prop}[c.f. Lemma 3.1 of \cite{W2}]\label{normHA}
For any given $R$, there exist positive constants $C$ and $\epsilon$ such that for any basis $\underline{s}=\{s_{1}, \dots ,s_{N_k}\}$ for $H^{0}(X,E(k))$ with $R$-bounded geometry and any hermitian matrix
$A$, we have $$\norm{H_{A}}_{C^m}^2 \leq Ck^{2n+m}  \norm{A}^2,$$
where $\norm{A}$ stands for the Hilbert-Schmidt norm.
\end{prop}

\begin{proof}
 Under the $R$-boundedness assumption, the classical $L^2$ techniques show that any holomorphic section $S$ of $E(k)$ satisfies for any integer $m\geq 0$,
$$\Vert \nabla^m S(x) \Vert^2_{C^0} \leq ck^{m+n}\int_X \Vert S\Vert^2 \frac{\omega^n}{n!}.$$ In particular, since $\{s_i\}$ forms a basis,
$\sum_{i} \Vert \nabla^m s_i(x) \Vert^2_{C^0} \leq crVk^{n+m}$. Now, pointwisely, one has applying successively Cauchy-Schwarz,
\begingroup
\allowdisplaybreaks   %julien added
\begin{align*}
\vert \nabla^m H_A\vert^2 = &  \vert \sum_{\alpha,\beta} a_{\alpha,\beta} \nabla^m S_\alpha S_\beta^*\vert^2\\
  = &  \vert \sum_{\alpha,\beta} a_{\alpha,\beta} \sum_l \nabla^l S_\alpha \nabla^{m-l} S_\beta^*\vert^2\\
 \leq & \left(\sum_l \sum_\beta \sum_\alpha \vert a_{\alpha,\beta} \vert \vert \nabla^l S_\alpha\vert \vert \nabla^{m-l} S_\beta^*\vert\right)^2\\
 \leq & \left( \sum_l \left( \sum_\beta \left( \sum_\alpha \vert a_{\alpha,\beta}\vert \vert \nabla^l S_\alpha\vert\right)^2 \right)^{\frac{1}{2}} \left(\sum_{\beta} \vert \nabla^{m-l} S_\beta\vert^2\right)^{\frac{1}{2}}\right)^{\hspace{-0.1cm}2}\\  
  \leq & \left(\sum_l \left( \sum_{\beta} (\sum_{\alpha} \vert a_{\alpha\beta}\vert^2)(\sum_{\alpha} \vert \nabla^l S_\alpha\vert^2)\right)^{\frac{1}{2}}\hspace{-0.2cm} \left( \sum_\beta \vert \nabla^{m-l} s_\beta\vert^2\right)^{\frac{1}{2}}\right)^{\hspace{-0.1cm}2}\\
\leq & \left( \sum_l  \left(\sum_{\alpha,\beta} \vert a_{\alpha\beta}\vert^2 \right)^{\frac{1}{2}}\hspace{-0.1cm} \left(\sum_{\alpha} \vert \nabla^l S_\alpha\vert^2\right)^{\frac{1}{2}}       \hspace{-0.2cm}\left( \sum_\beta \vert \nabla^{m-l} s_\beta\vert^2\right)^{\frac{1}{2}} \right)^{\hspace{-0.1cm}2} \\
\leq & \Vert A\Vert^2  \left(\sum_l \left(\sum_\alpha \vert \nabla^l S_\alpha\vert^2 \right)^{\frac{1}{2}} \left(\sum_{\beta} \vert \nabla^{m-l} S_\beta^*\vert^2\right)^{\frac{1}{2}} \right)^{\hspace{-0.1cm}2}\\
\leq & c\Vert A\Vert^2 k^{2n+m},
\end{align*}
\endgroup
where $c$ depends only on $m$.

\end{proof}

\begin{lem}\label{lemma2A}
Let $\widetilde{h}(t), \, \, t \in [0,1],$ be a smooth family of Fubini-Study metrics on $E(k)$, i.e. there exists hermitian inner product $H(t)$ on $H^{0}(X,E(k))$ such that
$\widetilde{h}(t)=\FS(H(t))\otimes \sigma^k.$ If all metrics $\widetilde{h}(t)$ have $R$-bounded geometry, then $$\norm{\widetilde{h}(1)-\widetilde{h}(0)}_{C^m} \leq Ck^{\frac{3n+m}{2}} d_{k}(H(0),H(1)).$$
\end{lem}

\begin{proof}

Let $\widetilde{h}(t)=e^{\phi_{t}}\widetilde{h}(0)$ and $A(t):=\frac{d}{dt}H(t)$.
Then $\frac{d}{dt}\phi_{t}=H_{A(t)}.$ Therefore,
$$\norm{\widetilde{h}(t)^{-1}\frac{d\widetilde{h}(t)}{dt}}_{C^m}=\norm{H_{A(t)}}_{C^m} \leq    Ck^{n+\frac{m}{2}}\norm{A(t)}.$$ 
Thus,

\begin{align*}
\norm{\widetilde{h}(1)-\widetilde{h}(0)}_{C^m} \leq& C_{R} \int_{0}^{1}\norm{\widetilde{h}(t)^{-1}\frac{d\widetilde{h}(t)}{dt}}_{C^m}dt \\
\leq& C_{R}Ck^{n+\frac{m}{2}}\int_{0}^{1}\norm{A(t)}dt\\
=& C_{R}Ck^{\frac{3n+m}{2}} d_{k}(H(0),H(1)).
\end{align*}
Note that $\widetilde{h}(t)$ is bounded from below.
\end{proof}

\begin{lem}\label{lemma2}
Let $H_{k} \in \mathcal{B}_{k}$ be a sequence of metrics with $\frac{R}{2}$-bounded geometry. Then, there exists a constant $C$ such that if $\widetilde{H} \in \mathcal{B}_{k}$ and  $$k^{\frac{4n+m}{2}}d_{k}(H_{k}, \widetilde{H}) \to 0 \,\, \textrm{as} \,\, k \to \infty,$$ then $\widetilde{H}$ has 
$R$-bounded geometry in $C^m$ topology.
\end{lem}

\begin{proof}
Let $\{s_{1}, \dots ,s_{N_k}\}$ be an orthonormal basis  for $H^{0}(X,E(k))$ with respect to $H_{k}$. Without loss of generality, we may assume that $$\inner{s_{i},s_{j}}_{\widetilde{H}}=e^{\lambda_{i}}\delta_{ij}.$$   Therefore, $$k^{n/2}d_{k}(H_{k}, \widetilde{H})=\left(\sum_{i=1}^{N_{k}} (e^{\lambda_{i}}-1)^2\right)^{1/2}\geq \max_{\alpha=1,\dots,N}\vert e^{\lambda_\alpha}-1\vert.$$ Define $\widetilde{h}=\FS(\widetilde{H})\otimes \sigma^k$ and $\widetilde{h_{k}}=\FS(\widetilde{H_{k}})\otimes \sigma^k$. Let $\widetilde{h}=e^{\phi}\widetilde{h_{k}}$. Therefore, $$Id_{E}=\sum s_{i}\otimes s_{i}^{*_{\widetilde{h}}},$$
$$e^{\phi}=\sum e^{\lambda_{i}}s_{i}\otimes s_{i}^{*_{\widetilde{h}}}.$$ Hence, applying Proposition \ref{normHA} to the matrix $(A)_{\alpha\beta}=(e^{\lambda_\alpha}-1)\delta_{\alpha,\beta}$, we get
\begin{align*}\norm{e^{\phi}-Id_{E}}_{C^m(\widetilde{h}_{0})}&=\norm{\sum_{i=1}^N(e^{\lambda_{i}}-1)s_{i}\otimes s_{i}^{*_{\widetilde{h}}}}_{C^m(\widetilde{h}_{0})}, \\&\leq  C\sqrt{N_{k}}\max_{\alpha=1,\dots, N}\vert e^{\lambda_{\alpha}}-1\vert k^{n+m/2},\\
 &\leq Ck^{\frac{4n+m}{2}}d_{k}(H_{k}, \widetilde{H}).
\end{align*}
 The conclusion follows.
 
 %It suffices to show that $\widetilde{h} \geq \frac{1}{2}h_{k}$ and $\norm{\widetilde{h}-h_{k}}_{C^m(\widetilde{h}_{0})} \leq \frac{R}{2}$ since $h_{k}\geq \frac{2}{R}\widetilde{h}_{0}$ and $\norm{\widetilde{h}_{0}-h_{k}}_{C^m(\widetilde{h}_{0})} \leq \frac{R}{2}$. 

\end{proof}

%%%%%%%%%%%%%%%%%%%%%%%%%%%

\begin{thm}\label{key}
Let $m$ be a nonnegative integer and consider $H_{k}(t)$, $H^{\prime}_{k}(t)$ two sequences of hermitian inner products on $H^{0}(X,E(k))$. Let $h_{k}(t)=\FS{(H_{k}(t))}$ and $h_{k}^{\prime}(t)=\FS{(H_{k}^{\prime}(t))}$. Suppose that 
$$d_{k}(H_{k}(t), H_{k}^{\prime}(t))=O(k^{-q}),$$  for a positive number $q >\frac{4n+m}{2}$. If $$h_{k}^{\prime}(t) \to h(t) \,\,\,\text{ as }k\to \infty \,\,\,\, \text{ in } \, C^{\infty},$$ uniformly in $t$, then
$$h_{k}(t) \to h(t) \,\,\,\text{ as } k\to \infty \,\,\,\, \text{ in } \, C^{m},$$ uniformly in $t$.

\end{thm}

\begin{proof}

It suffices to show that,
$$\norm{h_{k}(t)-h_{k}^{\prime}(t)}_{C^m} \to 0  \,\,\,\, \textrm{as } \,\,\,\, k \to \infty,$$ uniformly in $t$. Fix $T>0$ and in the rest of the proof, we assume that $ t \in [0,T].$
Let $H_{k}(t,s), \, 0\leq s \leq 1$ be the geodesic joining $H_{k}(t)$ and $H_{k}^{\prime}(t)$ in $\mathcal{B}_{k}$. We first show that there exists a positive number $R$ such that for all $t \in [0,T]$ and $s \in [0,1]$,
$H_{k}(t,s)$ has $R$-bounded geometry. There exists $R$ such that for all $t \in [0,T]$,
\begin{itemize}
\item $\displaystyle h(t) \geq \frac{3}{R}h_{0},$
\item $\displaystyle \norm{h(t) -h_{0}}_{C^m(h_{0})} \leq \frac{R}{3},$
\end{itemize}
since $\{h(t)| t\in [0,T]\}$ is a compact set. This together with the facts that $h_{k}^{\prime}(t) \to h(t) $ uniformly in $t$, and $h_{0}\otimes \sigma^k \sim \widetilde{h}_{0}$ imply that $h_{k}^{\prime}(t)\otimes \sigma^k$ has $\frac{R}{2}$ -bounded geometry. 
On the other hand, $$d_{k}(H_{k}(t,s), H_{k}^{\prime}(t) ) \leq d_{k}(H_{k}(t), H_{k}^{\prime}(t) )=O(k^{-q}).$$ Therefore, Lemma \ref{lemma2} implies that $H_{k}(t,s) $ has $R$ -bounded geometry. Hence, Lemma \ref{lemma2A} implies that 
\[\norm{h_{k}(t)-h_{k}^{\prime}(t)}_{C^m} \leq Ck^{\frac{3n+m}{2}} d_{k}(H_{k}(t),H_{k}^{\prime}(t)) 
\mathrel{\mathop{\longrightarrow}_{{k\rightarrow +\infty}}} 0,\]
since $d_{k}(H_{k}(t),H_{k}^{\prime}(t))=O(k^{-q})$ for $q >\frac{4n+m}{2}.$

%Then, assuming $(A1)$ and $(A2)$ we have, \[\norm{h_{k}(t)-h_{k}^{\prime}(t)}_{C^m} \leq C \max_{s,t} \norm{\widetilde{\mu}(H_{k}(t,s))}_{op} d_{k}(H_{k}(t),H_{k}^{\prime}(t)) \mathrel{\mathop{\longrightarrow}_{{k\rightarrow +\infty}}} 0,\]
%since $d_{k}(H_{k}(t),H_{k}^{\prime}(t))=O(k^{-q})$ for $q \gg 0.$ Hence the theorem holds under $(A1)$ and $(A2)$.\\
%Let us show that $(A2)$ is true. Lemma \ref{barmu} implies that $\norm{\bar{\mu}(H_{k}^{\prime}(t))}_{op}$ is bounded. We obtain 
%$$\norm{\bar{\mu}_{k}(H_{k}(t,s))}_{op} \leq e^{2d} \norm{\bar{\mu}_{k}(H_{k}(t,1))}_{op}\leq C,$$ for a certain uniform constant $C>0$. \todo[inline]{and $2d=2d_k(H_k(t,s),H_k(t,1))$ using Proposition \ref{normbarmu} isn't it?} On the other hand, the metrics $H_{k}(t,s)$ have $R$-bounded geometry. Actually, $h_{k}^{\prime}(t)$ has $\frac{R}{2}$-bounded geometry and  $$d(H_{k}(t,s), H_{k}^{\prime}(t)) \leq d(H_{k}(t), H_{k}^{\prime}(t)) =O(k^{-q}).$$ We can apply Lemma \ref{lemma2} with $q$ sufficiently large.
%Note that all estimates are uniform for $t \in [0,T]$. This shows that $(A1)$ is true.

\end{proof}
\color{black}
\section{Main results\label{mainsect}}

\begin{thm} \label{thm1}
Consider $h(t)$ solution of the modified Donaldson heat flow of endomorphisms given by system
\begin{align*}
\left\{\begin{array}{ll}
h(t)^{-1}\frac{dh(t)}{dt}&=-\left( \frac{\sqrt{-1}}{2\pi} \Lambda F_{(E,h(t))}+ \frac{1}{2}S(\omega) Id_{E}\right.\\
&\hspace{1cm}\left. -\frac{1}{rV}\int_{X}\tr \Big( \frac{\sqrt{-1}}{2\pi} \Lambda F_{(E,h(t))}+ \frac{1}{2} S(\omega) Id_{E}\Big)\frac{\omega^n}{n!}\right)Id_E,\\
h(0)&=h.
\end{array}\right.
\end{align*}
Let $H_{k}(t)$ be the normalized balancing flow
\begin{align*}
\left\{\begin{array}{ll}
\frac{dH_{k}(t)}{dt}&=-k^{n+1}\bar{\mu}_{0}(H_{k}(t))\\
H_{k}(0)&=H_{k}=\Hilb(h).
\end{array}\right.
\end{align*}
Let $h_{k}(t)=\FS(H_{k}(t))$. Then for all $t<+\infty$, the metric $h_k(t)$ converges in $C^\infty$ topology to $h(t)$ and this convergence is
 $C^1$ in $t$. Moreover, for any positive $T$, the convergence is uniform in $t$ for $t \in [0,T].$
\end{thm}
\begin{proof}
For any $q>0$, we constructed a sequence of approximating flows $h(k;t)$ such that:
\begin{itemize}
 \item $h(k;t)\rightarrow h(t)$ in $C^\infty$ topology, as $k\rightarrow +\infty$, with $h(t)$ solution of the modified Donaldson heat flow.
\item $d_k(H(k;t),\hat{H}(k;t))=O(k^{-q})$, where $h_k(t)=\FS(H(k;t))$ is the balancing flow, $h(k;t)=\FS(\hat{H}(k;t))$. 
\end{itemize}
We just need to prove that $\Vert h_k(t)-h(t)\Vert_{C^m} \rightarrow 0$ when $k\rightarrow +\infty$, for any positive integer $m$. 
We apply Theorem \ref{key} with data $H(k;t)$ and $\hat{H}(k;t)$. For $q$ large enough, we get the $C^m$ convergence of the balancing flow towards the modified Donaldson heat flow.\\
It is not difficult to check that the convergence is uniform in $t$. The crucial point is that the asymptotics of $Q_k$ and the Bergman kernel operators are uniform when $t$ varies, see for instance Theorem \ref{Qkthm}. The constructed metrics $h(k;t)$ converges also uniformly towards $h(t)$ since there are only  a finite number of perturbations. Proposition \ref{prop1} shows that the derivative of the involved metric along the balancing flow converges towards the $A_1(h(t))-\bar{A}_1(h(t))$ when $k\rightarrow +\infty$ smoothly. As it is uniform in $k$ for similar reasons as previously, one gets the $C^1$ convergence as expected.
\end{proof}

Of course, it is natural to ask if a modified balancing flow converges towards the classical Donaldson heat flow.

\begin{cor}\label{cor1}
 Consider $h(t)$ solution of Donaldson heat flow of endomorphisms given by system
\begin{align*}
\left\{\begin{array}{ll}
h(t)^{-1}\frac{dh(t)}{dt}&=-\left( \frac{\sqrt{-1}}{2\pi} \Lambda F_{(E,h(t))} -\mu(E) Id_E\right ),\\
h(0)&=h.
\end{array}\right.
\end{align*}
Let $H_{k}(t)$ be the normalized balancing flow
\begin{align*}
\left\{\begin{array}{ll}
\frac{dH_{k}(t)}{dt}&=-k^{n+1}\bar{\mu}_{0}(H_{k}(t))\\
H_{k}(0)&=H_{k}=\Hilb(h).
\end{array}\right.
\end{align*}
Let $\widehat{h}_{k}(t)=\FS(H_{k}(t))e^{\theta}$ with $\Delta_\omega \theta =-\frac{1}{2}\left(S(\omega) -\frac{1}{V}\int_{X}S(\omega)\frac{\omega^n}{n!}\right)$. Then for all $t<+\infty$, the metric $\widehat{h}_k(t)$ converges in $C^\infty$ topology to $h(t)$ and this convergence is  $C^1$ in $t$.
\end{cor}

\begin{proof}
 This is just a conformal change of the metric.  The result is given by noticing that $\theta$ is independent of $t$ and $ \Lambda F_{(E,h(t)e^{\theta})}= \Lambda F_{(E,h(t))} + \Delta \theta Id_E.$
\end{proof}

For numerical applications, it is pretty difficult to compute the scalar curvature of a metric as it involves 4th derivative in the potential of the metric. But, one can notice as a consequence of CTYZ expansion, that with $\omega=c_1(\sigma)$, the term $\frac{1}{2}\left(S(\omega) -\frac{1}{V}\int_{X}S(\omega)\frac{\omega^n}{n!}\right)$ is the limit of $\frac{\rho_k(\sigma)-h^0(L^k)}{Vk^{n-1}}$ when $k\rightarrow +\infty$. Here $\rho_k(\sigma)\in C^\infty(X,\mathbb{R}_+)$ is the Bergman function associated to the hermitian metric $\sigma$ and the $L^2$ inner product induced by $(\sigma,\frac{\omega^n}{n!})$. Consequently, it is natural to consider the following modification of the previous corollary.

\begin{cor}\label{cor2}
 Consider $h(t)$ solution of Donaldson heat flow of endomorphisms given by system
\begin{align*}
\left\{\begin{array}{ll}
h(t)^{-1}\frac{dh(t)}{dt}&=-\left( \frac{\sqrt{-1}}{2\pi} \Lambda F_{(E,h(t))} -\mu(E) Id_E\right ),\\
h(0)&=h.
\end{array}\right.
\end{align*}
Consider the volume form $\Omega'=\frac{\omega^n}{n!}\left(1+\frac{\rho_k(\sigma)-h^0(L^k)}{Vk^{n}}\right)$ for $k>>0$, and $\bar{\mu}'_{0}$ the moment map associated to $\Omega'$ instead of $\frac{\omega^n}{n!}$.
Let $H_{k}(t)$ be the normalized balancing flow
\begin{align*}
\left\{\begin{array}{ll}
\frac{dH_{k}(t)}{dt}&=-k^{n+1}\bar{\mu}'_{0}(H_{k}(t))\\
H_{k}(0)&=H_{k}=\Hilb(h).
\end{array}\right.
\end{align*}
Let $\widehat{h}_{k}(t)=\FS(H_{k}(t))$. Then for all $t<+\infty$, the metric $\widehat{h}_k(t)$ converges in $C^\infty$ topology to $h(t)$ and this convergence is  $C^1$ in $t$.
\end{cor}
\begin{proof}
 First notice that $\Omega'=\Omega(1+ O(k^{-1}))$ and $\Omega'$ does not depend on the time parameter $t$. Set $f_k=\frac{\rho_k(\sigma)-h^0(L^k)}{Vk^{n}}=O(k^{-1})$. 
The main change when introducing $\Omega'$ is coming from Proposition \ref{prop1}. One can check that with this change of volume form, one has 
 \begin{align*}
\tr(\bar{\mu}_0'(H_k)\tilde{\mu}(H_k))=&  \tr(\bar{\mu}_0'(H_k)\tilde{\mu}(H_k)) \\
&-\sum_{i,j=1}^{N_k} \int_X \langle s_i, (f_k Id_E+O(\frac{1}{k})) s_j \rangle_{h\otimes \sigma^k}\frac{\omega^n}{n!}\langle .,s_i\rangle_{h\otimes \sigma^k}s_j\\
& +O(k^{-2-n}),\\
=&k^{-n-1}(\widetilde{A}_1(h)-kQ_k(f_k Id_E)) + O(k^{-n-2}),\\
=&k^{-n-1}(\widetilde{A}_1(h)-k f_k Id_E) + O(k^{-n-2}),\\
=&k^{-n-1}\left(\widetilde{A}_1(h) - \frac{1}{2}\left(S(\omega) -\frac{1}{V}\int_{X}S(\omega)\frac{\omega^n}{n!}\right)\right) \\
&+ O(k^{-n-2}), \\
=&k^{-n-1}\left(\frac{\sqrt{-1}}{2\pi} \Lambda F_{(E,h)} -\mu(E) Id_E\right)+ O(k^{-n-2}).
\end{align*}
Here we have used the asymptotic of the $Q_k$ operator and CTYZ expansion. This shows that the limit of the normalized balancing flow associated to $\Omega'$ is Donaldson heat flow.
\end{proof}
\begin{thm}\label{thm2} Let us fix $h_{[0]}$ a hermitian metric on $E$. 
 Let $h_{[m],k}$ be the $m$-th image by the map $\Phi_k=\FS \circ \Hilb $ of the hermitian metric $h_{[0]}$ on $E\otimes L^k$, so that  
$$h_{[m],k}= \Phi_k^{(m)}(h_{[0]}).$$ 
Assume that $m=m(k)$ and $m/k\rightarrow t\in \mathbb{R}_+$ as $k\rightarrow +\infty$. Then $h_{[m],k}$ converges in $C^\infty$ topology towards $h(t)$ solution of the flow
\begin{align*}
\left\{\begin{array}{ll}
h(t)^{-1}\frac{dh(t)}{dt}&=-\left( \frac{\sqrt{-1}}{2\pi} \Lambda F_{(E,h(t))}+ \frac{1}{2}S(\omega) Id_{E}\right.\\
&\hspace{1cm}\left. -\frac{1}{rV}\int_{X}\tr \Big( \frac{\sqrt{-1}}{2\pi} \Lambda F_{(E,h(t))}+ \frac{1}{2} S(\omega) Id_{E}\Big)\frac{\omega^n}{n!}\right)Id_E,\\
h(0)&=h_{[0]} . 
\end{array}\right.
\end{align*}
\end{thm}

\begin{proof}
 Let $m\geq 0$. First, using mean value theorem,
$$\frac{h\left(\frac{m+1}{k}\right)-h\left(\frac{m}{k}\right)}{1/k} = \frac{dh(t)}{dt}_{\vert t=\frac{m}{k}} + O(1/k).$$
But by assumption, this gives 
\begin{equation}\label{mean1}{h\left(\frac{m+1}{k}\right)-h\left(\frac{m}{k}\right)} = -\frac{1}{k}h\left(\frac{m}{k}\right)\widetilde{A}_1\left(h\left(\frac{m}{k}\right)\right) + O(1/k^2).\end{equation}
By CTYZ expansion,  we know that
\begin{equation}\label{mean2}\FS\circ \Hilb\left(h\left(\frac{m}{k}\right)\right) = h\left(\frac{m}{k}\right)\left(Id_E-\frac{1}{k}\widetilde{A}_1\left(h\left(\frac{m}{k}\right)\right)\right) + O(k^{-2}).\end{equation}
Let us show by induction that \begin{equation}\label{induc}\Vert h_{[m],k}- h(\frac{m}{k})\Vert_{C^m} \leq C_0\frac{m}{k^2},\end{equation}
where the norm is taken with respect to a fixed background metric.
This is true for $m=0$ from assumption. Now, at step $m+1$, one has using successively \eqref{mean1}, \eqref{mean2},
\begin{align*}
\Big\Vert h_{[m+1],k} -h\left(\frac{m+1}{k}\right)\Big\Vert_{C^m}\hspace{-0.2cm} =& \Big\Vert \FS\circ \Hilb (h_{[m],k})- h\left(\frac{m+1}{k}\right)\Big\Vert_{C^m},\\
\leq&\Big\Vert \FS\circ \Hilb (h_{[m],k})  \\
&   - h\left(\frac{m}{k}\right)\left(Id_E -\frac{1}{k}\widetilde{A}_1\left(h\left(\frac{m}{k}\right)\right)\right) \Big\Vert_{C^m}\\
&+C/k^2,\\
\leq &\Big\Vert  \FS \circ  \Hilb (h_{[m],k}) \\
&- \FS\circ\Hilb  \left(h\left(\frac{m}{k}\right)\right)\Big\Vert_{C^m} \\
&+ C'/k^2 .
\end{align*}
Here the constants $C,C'$ depends on $h(t)$ and its covariant derivatives for $t\in [m/k, (m+1)/k]$. Since $m/k\rightarrow t$ and $h(t)$ is a smooth family, the estimates above are uniform when $m,k$ vary for $k$ sufficiently large. Thus, up to taking $C'$ a bit larger, we can assume that $C'$ depends only on $t$ for $k$ large enough. We claim given any hermitian metric $h_0,h_1$, one has
$$ \Vert \FS\circ \Hilb (h_1) - \FS\circ \Hilb  (h_0) \Vert_{C^m} \leq \Vert h_0- h_1\Vert_{C^m}.$$ 
With the claim in hands, we obtain using the induction at step $m$, that 
$$\Big\Vert h_{[m+1],k} -h\left(\frac{m+1}{k}\right)\Big\Vert_{C^m} \leq \frac{C_0m}{k^2} +\frac{C'}{k^2}\leq C_0\frac{m+1}{k^2},$$
since we can take $C_0>C'$. This proves \eqref{induc} at step $m+1$ and thus the expected convergence.\\
Let us prove the claim now.  First, we notice that if $h_1\geq h_0$ (which means as the level of endomorphisms that the difference $h_1-h_0$ is a positive semi-definite  hermitian matrix, then 
\begin{equation}\label{first1} \FS\circ \Hilb(h_0)\leq \FS\circ \Hilb(h_1). \end{equation} Actually, with $S_0=(s_{i,0})_{i=1,..,N}$ an orthonormal basis with respect to the metric $\Hilb(h_0)$, we have
\begin{align}
\FS\circ \Hilb(h_0)&= \frac{N_k}{rV} h_0 \left( \sum_{i=1}^{N_k} s_{i,0} \otimes s_{i,0}^{*_{h_0}}\right)^{-1},\nonumber \\
&= \frac{N_k}{rV} h_1 \left( \sum_{i=1}^{N_k} s_{i,0} \otimes s_{i,0}^{*_{h_1}}\right)^{-1} \label{first2}
\end{align}
Because we are dealing hermitian positive definite matrices, there exist from simultaneous diagonalization techniques $\sigma\in GL_N(C)$ and $S_1=(s_{i,1})_{i=1,..,N}$ an orthonormal basis with respect to $\Hilb(h_1)$, such that $S_0=\sigma \cdot S_1$ and $\sigma$ is diagonal. Thus $\{s_{i,1}\}_{i=1,..,N}$ is orthogonal with respect to $\Hilb(h_0)$ and we can write $s_{i,0}=\delta_i^{1/2} s_{i,1}$ where by assumption we have 
$$\delta_i:=\Vert s_{i,0}\Vert^2_{\Hilb(h_1)}\geq \Vert s_{i,0}\Vert^2_{\Hilb(h_0)}=1.$$
Consequently, $$\sum_{i=1}^{N_k} s_{i,0} \otimes s_{i,0}^{*_{h_1}}= \sum_{i=1}^{N_k} \delta_i s_{i,1} \otimes s_{i,1}^{*_{h_1}}\geq \sum_{i=1}^{N_k} s_{i,1} \otimes s_{i,1}^{*_{h_1}}>0,$$
which induces at the level of endomorphisms,
$$\left( \sum_{i=1}^{N_k} s_{i,0} \otimes s_{i,0}^{*_{h_1}}\right)^{-1} \leq \left( \sum_{i=1}^{N_k} s_{i,1} \otimes s_{i,1}^{*_{h_1}}\right)^{-1}$$
and thus with \eqref{first2}, we obtain \eqref{first1}.

Let us choose $c\geq 1$ the smallest real constant such that $h_0\leq c h_1$ and $h_1\leq c h_0$. Then $\Vert h_0-h_1\Vert_{C^m}=c-1$ and we can apply the previous inequality \eqref{first1} to successively $(h_0,ch_1)$ and $(h_1,ch_0)$. We obtain
\begin{align*}
 \FS\circ \Hilb(h_0) \leq \FS\circ \Hilb(ch_1) = c \FS\circ \Hilb(h_1)\\
 \FS\circ \Hilb(h_1) \leq \FS\circ \Hilb(ch_0) = c \FS\circ \Hilb(h_0)
\end{align*}
Therefore $\Vert \FS\circ \Hilb (h_1) - \FS\circ \Hilb  (h_0)\Vert_{C^m} \leq c-1 = \Vert h_0-h_1\Vert_{C^m}$ and thus the claim is proved.
\end{proof}

In particular, it is possible to approximate Donaldson heat flow and the Yang-Mills flow by iterative methods using the volume form defined in Corollary \ref{cor2} (see also Remark \ref{rmk1}).

% ----------------------------------------------------------------

\end{document}